\documentclass[12pt]{article} 

\title{High order Bellman equations and weakly chained diagonally
dominant tensors}
\author{%
  Parsiad Azimzadeh%
  \thanks{Department of Mathematics, University of Michigan
  ({\tt parsiad@umich.edu}, {\tt erhan@umich.edu}).}
  \and
  Erhan Bayraktar\footnotemark[1] %
  \thanks{E. Bayraktar is supported in part by the National Science
  Foundation under grant DMS-1613170.}.%
}

\makeatletter

\usepackage{amsfonts}   
\usepackage{amssymb}    
\usepackage{booktabs}   
\usepackage{enumitem}   
\usepackage{multirow}   
\usepackage{nicefrac}   
\usepackage{subfig}     
\usepackage{verbatim}   

\usepackage{algorithm}
\usepackage{algpseudocode}

\usepackage{tikz}
\usetikzlibrary{arrows,decorations.pathreplacing,shapes}

\@ifclassloaded{siamart1116}{

  \newsiamthm{example}{Example}
  \newsiamthm{rem}{Remark}

  \headers{High order Bellman equations and w.c.d.d. tensors}{Parsiad
  Azimzadeh and Erhan Bayraktar}

  \bibliographystyle{plainnat}

  \newcommand{\qedhere}{}

  \newcommand{\figwidth}{4in}
  \newcommand{\figscale}{1}
  \newcommand{\algorithmcodesize}{}

}{

  \date{}

  \pdfoutput=1

  \usepackage{microtype}
  \usepackage[T1]{fontenc}
  \usepackage{lmodern}

  \usepackage{amsthm}
  \usepackage{amsmath}
  \usepackage[numbers]{natbib}

  \usepackage[margin=1in]{geometry}

  \renewcommand*{\leq}{\leqslant}
  \renewcommand*{\geq}{\geqslant}

  \newcounter{dummy}
  \newtheorem{corollary}[dummy]{Corollary}
  \newtheorem{definition}[dummy]{Definition}
  
  \newtheorem{lemma}[dummy]{Lemma}
  \newtheorem{proposition}[dummy]{Proposition}
  \newtheorem{rem}[dummy]{Remark}
  \newtheorem{theorem}[dummy]{Theorem}

  \usepackage{hyperref}
  \bibliographystyle{plainnat}

  \usepackage{prettyref}
  \newcommand{\cref}{\prettyref}
  \newcommand{\Cref}{\prettyref}
  \newrefformat{alg}{Algorithm \ref{#1}}
  \newrefformat{app}{Appendix \ref{#1}}
  \newrefformat{cor}{Corollary \ref{#1}}
  \newrefformat{def}{Definition \ref{#1}}
  \newrefformat{exa}{Example \ref{#1}}
  \newrefformat{fig}{Figure \ref{#1}}
  \newrefformat{lem}{Lemma \ref{#1}}
  \newrefformat{prop}{Proposition \ref{#1}}
  \newrefformat{rem}{Remark \ref{#1}}
  \newrefformat{sec}{Section \ref{#1}}
  \newrefformat{subsec}{Section \ref{#1}}
  \newrefformat{subsubsec}{Section \ref{#1}}
  \newrefformat{tab}{Table \ref{#1}}
  \newrefformat{thm}{Theorem \ref{#1}}

  \newcommand*{\email}[1]{\href{mailto:#1}{\nolinkurl{#1}} }

  \usepackage{xcolor}
  \definecolor{mylinkcolor}{HTML}{0066cc}
  \definecolor{mycitecolor}{HTML}{008800}
  \definecolor{myurlcolor}{HTML}{0066cc}

  \hypersetup{
    linkcolor  = mylinkcolor,
    citecolor  = mycitecolor,
    urlcolor   = myurlcolor,
    colorlinks = true,
  }

  \providecommand{\doi}{}
  \renewcommand{\doi}[1]{\href{https://doi.org/#1}{doi:
  \begingroup\urlstyle{rm}\nolinkurl{#1}\endgroup}}

  \newcommand{\figwidth}{5.75in}
  \newcommand{\figscale}{1.125}

  \newcommand{\algorithmcodesize}{\small}
}

\makeatother

\begin{document}

\maketitle

\begin{abstract}
  We introduce \emph{high order Bellman equations}, extending
  classical Bellman equations to the tensor setting.
  We introduce \emph{weakly chained diagonally dominant} (w.c.d.d.)
  tensors and show that a sufficient condition for the existence and
  uniqueness of a positive solution to a high order Bellman equation
  is that the tensors appearing in the equation are w.c.d.d. M-tensors.
  In this case, we give a \emph{policy iteration} algorithm to
  compute this solution.
  We also prove that a weakly diagonally dominant Z-tensor with
  nonnegative diagonals is a strong M-tensor if and only if it is w.c.d.d.
  This last point is analogous to a corresponding result in the
  matrix setting and tightens a result from {[}L. Zhang, L. Qi, and
    G. Zhou. ``M-tensors and some applications.'' \emph{SIAM Journal on
  Matrix Analysis and Applications} (2014){]}.
  We apply our results to obtain a provably convergent numerical
  scheme for an optimal control problem using an ``optimize then
  discretize'' approach which outperforms (in both computation time
  and accuracy) a classical ``discretize then optimize'' approach.
  To the best of our knowledge, a link between M-tensors and optimal
  control has not been previously established.
\end{abstract}

\section{Introduction}

In this work, we introduce and study the nonlinear problem
\begin{equation}
  \text{find }u\in\mathbb{R}^{n}\text{ such that
  }\min_{P\in\mathcal{P}}\left\{ A(P)u^{m-1}-b(P)\right\} =0\label{eq:bellman}
\end{equation}
where $A(P)$ is an $m$-order and $n$-dimensional real tensor, $b(P)$
is a real vector, $\mathcal{P}$ is a nonempty compact set, and the
minimum is taken with respect to the coordinatewise order on
$\mathbb{R}^{n}$ (see \ref{enu:coordinatewise} in \cref{sec:bellman}).

If $m=2$, then $A(P)$ is a square matrix and $A(P)u^{m-1}\equiv
A(P)u$ is the ordinary matrix-vector product.
In this case, \eqref{eq:bellman} is the celebrated Bellman equation
for optimal decision making on a Markov chain.
Aside from Markov chains, \eqref{eq:bellman} also
arises from discretizations of differential equations from optimal
control \cite{MR1217486}.

If $m>2$, then $A(P)u^{m-1}$ defines a vector whose $i$-th component
is a multivariate polynomial in the entries of $u$:
\[
  (A(P)u^{m-1})_{i}=\sum_{i_{2},\ldots,i_{m}}(A(P))_{ii_{2}\cdots
  i_{m}}u_{i_{2}}\cdots u_{i_{m}}
\]
(see also \eqref{eq:polynomial}).
As such, we refer to \eqref{eq:bellman} as a \emph{Bellman equation
of order $m$}.
We are motivated to study this equation since, as we will see in the
sequel, it arises from a so-called ``optimize then discretize''
\cite{MR2333222} scheme for a differential equation.

Our main goal is to characterize the existence and uniqueness of a
solution $u$ to \eqref{eq:bellman} and to obtain a fast and provably
convergent algorithm for computing it.
If $m=2$, existence and uniqueness is guaranteed when $A(P)$ is a
nonsingular M-matrix for each $P$ \cite[Theorem 2.1]{MR2551155}
(along with some other mild conditions
on the functions $A$ and $b$).
We obtain an analogous result for the $m>2$ case when $A(P)$ is a
strictly diagonally dominant (and hence nonsingular\footnote{The
    terms ``nonsingular M-tensor'' and ``strong M-tensor'' are synonymous
in the literature.}) \emph{M-tensor} for each $P$
(\cref{lem:uniqueness} and \cref{lem:existence}).
M-tensors, a generalization of M-matrices, were introduced in
\cite{MR3190753,MR3116429} in order to test the positive definiteness
of multivariate polynomials.

However, strict diagonal dominance is a rather strong condition.
In order to generalize our results, we extend the notion of
\emph{weakly chained diagonal dominance} from matrices to tensors
(\cref{def:wcdd}).
By restricting our attention to the case in which $\mathcal{P}$ is
finite, we establish existence and uniqueness of a solution $u$ to
\eqref{eq:bellman} under the weaker requirement that $A(P)$ is a
\emph{weakly chained diagonally dominant} M-tensor
(\cref{lem:existence_wcdd}) and give a policy iteration algorithm to
compute the solution (\cref{alg:policy_iteration}).
Analogously to the $m=2$ case, the assumed finitude of $\mathcal{P}$
is sufficient for practical applications, though whether this
assumption can be dropped remains an interesting open theoretical
question (\cref{rem:finite_disclaimer}).

We also establish the following result, which should be of broader
interest to the M-tensor community:
\begin{theorem}
  \label{thm:wcdd}
  Let $A$ be a weakly diagonally dominant Z-tensor with nonnegative diagonals.
  Then, the following are equivalent:
  \begin{enumerate}[label=(\emph{\roman*})]
    \item \label{enu:1}$A$ is a strong M-tensor.
    \item \label{enu:2}Zero is not an eigenvalue of $A$.
    \item \label{enu:3}$A$ is weakly chained diagonally dominant
      (\cref{def:wcdd}).
  \end{enumerate}
\end{theorem}
An analogous equivalence result for matrices was recently proved in
\cite{azimzadeh2017fast}.
Since a weakly irreducibly diagonally dominant tensor is a weakly
chained diagonally dominant tensor (\cref{lem:irreducible}), the
following is an immediate consequence:
\begin{corollary}
  \label{cor:wcdd}Let $A$ be a Z-tensor with nonnegative diagonals.
  If $A$ is weakly irreducibly diagonally dominant, then $A$ is a
  strong M-tensor.
\end{corollary}
Since an irreducible tensor is weakly irreducible
(\cref{cor:irreducible}), \cref{thm:wcdd} and \cref{cor:wcdd} can be
thought of as tightening the following result:
\begin{proposition}[{\cite[Theorem 3.15]{MR3190753}}]
  \label{prop:zhang_0}Let $A$ be a Z-tensor with nonnegative diagonals.
  If $A$ is strictly or irreducibly diagonally dominant, then $A$ is
  a strong M-tensor.
\end{proposition}
Moreover, \cref{thm:wcdd} yields a graph-theoretic characterization
of weakly diagonally dominant strong M-tensors and a fast algorithm
to determine if an arbitrary weakly diagonally dominant tensor is a
strong M-tensor (\cref{rem:algorithm}).

This work is organized as follows.
In \cref{sec:preliminaries}, we recall some standard definitions and
results for tensors.
In \cref{sec:wcdd}, we introduce the notion of a weakly chained
diagonally dominant tensor.
A proof of \cref{thm:wcdd} is given in \cref{sec:proof}.
In \cref{sec:bellman}, we study high order ($m>2$) Bellman equations.
In \cref{sec:application}, we use our results to study numerically
some problems from optimal stochastic control.

\section{\label{sec:preliminaries}Preliminaries}

For the convenience of the reader, we gather in this section some
definitions and well-known results (cf. \cite{MR3660696}) concerning
tensors of the form
\[
  A=(a_{i_{1}\cdots i_{m}}),\qquad a_{i_{1}\cdots
  i_{m}}\in\mathbb{R},\qquad1\leq i_{1},\ldots,i_{m}\leq n.
\]
Such tensors are called $m$-order and $n$-dimensional real tensors,
though we will simply refer to them as ``tensors'' in this work.
We call $a_{1\cdots1}$, $a_{2\cdots2}$, etc. the diagonal entries of $A$.
All other entries are referred to as off-diagonal.
\begin{definition}[\cite{MR2178089}]
  Let $x=(x_{1},\ldots,x_{n})$ be a vector in $\mathbb{C}^{n}$ and
  $A=(a_{i_{1}\cdots i_{m}})$ be a tensor.
  We denote by $x^{[\alpha]}=(x_{1}^{\alpha},\ldots,x_{n}^{\alpha})$
  the coordinatewise power of $x$ and by $Ax^{m-1}$ the vector in
  $\mathbb{C}^{n}$ whose $i$-th entry is
  \begin{equation}
    \sum_{i_{2},\ldots,i_{m}}a_{ii_{2}\cdots i_{m}}x_{i_{2}}\cdots
    x_{i_{m}}.\label{eq:polynomial}
  \end{equation}
  We call $\lambda$ in $\mathbb{C}$ an \emph{eigenvalue} of $A$ if we
  can find a vector $x$ in $\mathbb{C}^{n}\setminus\{0\}$ such that
  \[
    Ax^{m-1}=\lambda x^{[m-1]}.
  \]
  The vector $x$ is called an \emph{eigenvector} associated with $\lambda$.
  The spectrum $\sigma(A)$ of $A$ is the set of all eigenvalues of $A$.
  The spectral radius of $A$ is $\rho(A)=\max_{\lambda\in\sigma(A)}|\lambda|$.
\end{definition}
Z and M-tensors, defined below, are natural extensions of Z and M-matrices.
\begin{definition}[{\cite[Pg. 440]{MR3190753}}]
  A \emph{Z-tensor} is a real tensor whose off-diagonal entries are nonpositive.
\end{definition}
\begin{definition}[{\cite[Definition 3.1]{MR3190753}}]
  A tensor $A$ is an \emph{M-tensor} if there exists a nonnegative
  tensor $B$ (i.e., a tensor with nonnegative entries) and a real
  number $s\geq\rho(B)$ such that
  \[
    A=sI-B
  \]
  where $I$ is the identity tensor (i.e., the tensor with ones on its
  diagonal and zeros elsewhere).
  If $s>\rho(B)$, then $A$ is called a \emph{strong M-tensor}.
\end{definition}
Unlike the matrix setting, there are two distinct notions of
irreducibility for tensors, introduced in \cite{MR2996365}.
Both are given below.
\begin{definition}[{\cite[Pg. 739]{MR2996365}}]
  A tensor $A=(a_{i_{1}\cdots i_{m}})$ is \emph{reducible} if there
  exists a nonempty proper index subset
  $\Lambda\subsetneq\{1,\ldots,n\}$ such that
  \[
    a_{ii_{2}\cdots i_{m}}=0\qquad\text{if }i\in\Lambda\text{ and
    }i_{2},\ldots,i_{m}\notin\Lambda.
  \]
  Otherwise, we say $A$ is \emph{irreducible}.
\end{definition}
\begin{definition}[{\cite[Definition 2.2]{MR3146525}}]
  Let $A=(a_{i_{1}\cdots i_{m}})$ be a tensor and $R(|A|)$, the
  \emph{representation} of $A$, denote the $n\times n$ matrix whose
  $(i,j)$-th entry is given by
  \[
    \sum_{i_{2},\ldots,i_{m}}\left|a_{ii_{2}\cdots
    i_{m}}\right|\boldsymbol{1}_{\{i_{2},\ldots,i_{m}\}}(j)
  \]
  where $\boldsymbol{1}_{\mathcal{U}}$ is the indicator function of
  the set $\mathcal{U}$.
  We say $A$ is \emph{weakly reducible} if $R(|A|)$ is a reducible matrix.
  Otherwise, we say $A$ is \emph{weakly irreducible}.
\end{definition}
In \cite[Lemma 3.1]{MR2996365}, the authors show that irreducibility
is a stronger requirement than weak irreducibility.
We summarize this below, including what we believe to be a simpler
proof for the reader's convenience.
\begin{proposition}
  A tensor that is weakly reducible is reducible.
\end{proposition}
\begin{proof}
  Let $A=(a_{i_{1}\cdots i_{m}})$ be a weakly reducible tensor so
  that the $n\times n$ matrix $R(|A|)=(r_{ij})$ is reducible.
  Then, there exists a nonempty proper index subset
  $\Lambda\subsetneq\{1,\ldots,n\}$ such that
  \[
    r_{ij}=\sum_{i_{2},\ldots,i_{m}}\left|a_{ii_{2}\cdots
    i_{m}}\right|\boldsymbol{1}_{\{i_{2},\ldots,i_{m}\}}(j)=0\qquad\text{if
    }i\in\Lambda\text{ and }j\notin\Lambda.
  \]
  Therefore,
  \[
    a_{ii_{2}\cdots i_{m}}=0\qquad\text{if }i\in\Lambda\text{ and
    }i_{k}\notin\Lambda\text{ for some }2\leq k\leq m
  \]
  and hence $A$ is reducible.
\end{proof}
\begin{corollary}
  \label{cor:irreducible}An irreducible tensor is weakly irreducible.
\end{corollary}
We close this section with the notion of diagonal dominance.
\begin{definition}[{\cite[Definition 3.14]{MR3190753}}]
  Let $A=(a_{i_{1}\cdots i_{m}})$ be a tensor.
  We say that the $i$-th
  row of $A$ is \emph{strictly diagonally dominant} (s.d.d.) if
  \begin{equation}
    \left|a_{i\cdots i}\right|>\sum_{(i_{2},\ldots
    i_{m})\neq(i,\ldots,i)}\left|a_{ii_{2}\cdots i_{m}}\right|.\label{eq:sdd}
  \end{equation}
  We say $A$ is s.d.d. if all of its rows are s.d.d. \emph{Weakly
  diagonally dominant} (w.d.d.) is defined with weak inequality
  ($\geq$) instead.
  We use
  \[
    J(A) = \left \{
      1 \leq i \leq n
      \colon i \text{ satisfies } \eqref{eq:sdd}
    \right \}
  \]
  to denote the set of s.d.d. rows of $A$.
\end{definition}
\begin{definition}[{\cite[Definition 3.14]{MR3190753}}]
  \label{def:irreducibly_dominant}We say a tensor $A$ is
  \emph{(weakly) irreducibly diagonally dominant} if it is (weakly)
  irreducible, w.d.d., and $J(A)$ is nonempty.
\end{definition}

\section{\label{sec:wcdd}Weakly chained diagonally dominant tensors}

Before we introduce the notion of weakly chained diagonally dominant
(w.c.d.d.) tensors, we define the directed graph associated with a tensor.
\begin{definition}
  \label{def:graph}Let $A=(a_{i_{1}\cdots i_{m}})$ be a tensor.
  \begin{enumerate}[label=(\emph{\roman*})]
    \item The \emph{directed graph} of $A$, denoted
      $\operatorname{graph}A$, is a tuple $(V,E)$ consisting of the
      vertex set $V=\{1,\ldots,n\}$ and edge set $E\subset V\times V$
      satisfying $(i,j)\in E$ if and only if $a_{ii_{2}\cdots
      i_{m}}\neq0$ for some $(i_{2},\ldots,i_{m})$ such that
      $j\in\{i_{2},\ldots,i_{m}\}$.
    \item A \emph{walk} in $\operatorname{graph}A$ is a nonempty
      finite sequence
      $(i^{1},i^{2}),(i^{2},i^{3}),\ldots,(i^{k-1},i^{k})$ of
      ``adjacent'' edges in $E$.
  \end{enumerate}
\end{definition}
The proof of the next result, being a trivial consequence of the
above definition, is omitted.
\begin{lemma}
  \label{lem:graph}$\operatorname{graph}A=\operatorname{graph}R(|A|)$
  for any tensor $A$.
\end{lemma}
Since each vertex in the directed graph of a tensor corresponds to a
row $i$, we use the terms row and vertex interchangeably.
To simplify matters, we hereafter denote edges by $i\rightarrow j$
instead of $(i,j)$ and walks by $i^{1}\rightarrow
i^{2}\rightarrow\cdots\rightarrow i^{k}$
instead of $(i^{1},i^{2}),\ldots,(i^{k-1},i^{k})$.
We are now ready to define w.c.d.d.
\begin{definition}
  \label{def:wcdd}A tensor $A$ is \emph{w.c.d.d.} if all of the
  following are satisfied:
  \begin{enumerate}[label=(\emph{\roman*})]
    \item $A$ is w.d.d.
    \item $J(A)$ is nonempty.
    \item For each $i^{1}\notin J(A)$, there exists a walk
      $i^{1}\rightarrow i^{2}\rightarrow\cdots\rightarrow i^{k}$ in
      $\operatorname{graph}A$ such that $i^{k}\in J(A)$.
  \end{enumerate}
\end{definition}
If the tensor is of order $m=2$ (i.e., the tensor is a matrix), then
the above becomes the usual definition of w.c.d.d. for matrices
\cite[Definition 2.20]{azimzadeh2017fast}.
\begin{rem}
  \label{rem:algorithm}A trivial extension of the algorithm in
  \cite{azimzadeh2017fast} allows us to test if a weakly diagonally
  dominant tensor is weakly chained diagonally dominant with
  $O(n^{m})$ effort if the tensor is dense.
  Less computational effort is required if the tensor is sparse (we
  refer to \cite{azimzadeh2017fast} for details). %
\end{rem}
\begin{lemma}
  \label{lem:irreducible}A weakly irreducibly diagonally dominant
  tensor is a w.c.d.d. tensor.
\end{lemma}
\begin{proof}
  If $A$ is weakly irreducibly diagonally dominant, then $J(A)$ is
  nonempty and $\operatorname{graph}R(|A|)$ is strongly connected
  (i.e., for any pair of vertices $(i,j)$, there is a walk starting
  at $i$ and ending at $j$). The result then follows by \cref{lem:graph}.
\end{proof}

\section{\label{sec:proof}Proof of Theorem \ref{thm:wcdd}}

In the following, we denote by $\operatorname{Re}z$ the real part of
a complex number $z$.
We say a vector is positive if it lies in the positive orthant
$\mathbb{R}_{++}^{n}=(0,\infty)^{n}$.
Similarly, any element of $\mathbb{R}_{+}^{n}=[0,\infty)^{n}$ is
called a nonnegative vector.
Our proof relies on the following results:
\begin{proposition}[{\cite[Theorem 3.3]{MR3190753}}]
  \label{prop:zhang_1}$\min_{\lambda\in\sigma(A)}\operatorname{Re}\lambda$
  is nonnegative (resp. positive) whenever $A$ is an M-tensor (resp.
  strong M-tensor).
\end{proposition}
\begin{proposition}[{\cite[Theorem 3.15]{MR3190753}}]
  \label{prop:zhang_2}If $A$ is a w.d.d. Z-tensor with nonnegative
  diagonals, then $A$ is an M-tensor.
\end{proposition}
Our proof also relies on a corollary to the following result:
\begin{proposition}[{\cite[Pg. 697]{MR3519197}}]
  \label{prop:ding}Let $A$ be a strong M-tensor. For each positive
  vector $b$ (of compatible size), there exists a unique positive
  vector $x$ which solves the tensor equation $Ax^{m-1}=b$. Denoting
  by $A_{++}^{-1}:\mathbb{R}_{++}^{n}\rightarrow\mathbb{R}_{++}^{n}$
  the mapping from positive right hand sides $b$ to positive
  solutions $x$, $A_{++}^{-1}$ is nondecreasing with respect to the
  coordinatewise order.
\end{proposition}

The corollary, which is of independent interest, establishes
existence (but not uniqueness) of nonnegative solution $x$ to the
tensor equation $Ax^{m-1}=b$ when $b$ is a nonnegative vector and $A$
is a strong M-tensor.
\begin{corollary}
  \label{cor:ding_corollary}Let $A$ be a strong M-tensor. Then, there
  exists a nondecreasing map
  $A_{+}^{-1}:\mathbb{R}_{+}^{n}\rightarrow\mathbb{R}_{+}^{n}$ which
  associates to each nonnegative right hand side $b$ a nonnegative
  solution $x$ of the tensor equation $Ax^{m-1}=b$.
\end{corollary}
\begin{proof}
  For $k\geq1$, define
  \[
    x_{(k)}=A_{++}^{-1}(b+1/k)
  \]
  where $b+1/k$ is the vector obtained by adding $1/k$ to each entry
  of $b$. Since $A_{++}^{-1}$ is nondecreasing, the sequence $(x_{(k)})_{k}$
  is nonincreasing with respect to the coordinatewise order. Moreover,
  since each $x_{(k)}$ is a nonnegative vector, this sequence is bounded
  below by the zero vector and hence has a limit $x$. Taking
  $k\rightarrow\infty$
  in
  \[
    Ax_{(k)}^{m-1}=b+1/k
  \]
  and employing the continuity of the map $y\mapsto Ay^{m-1}$, we obtain
  $Ax^{m-1}=b$.

  The above implies that the map
  $A_{+}^{-1}:\mathbb{R}_{+}^{n}\rightarrow\mathbb{R}_{+}^{n}$
  \[
    A_{+}^{-1}(b)=\lim_{k\rightarrow\infty}A_{++}^{-1}(b+1/k)
  \]
  is well-defined and associates to each nonnegative vector $b$ a nonnegative
  solution $x$ of the tensor equation $Ax^{m-1}=b$. That $A_{+}^{-1}$
  is nondecreasing is an immediate consequence of $A_{++}^{-1}$ being
  nondecreasing.
\end{proof}
We require one last intermediate result, which captures the invariance
of spectra under permutation. It can be thought of as a generalization
of the fact that for any square matrix $A$ and permutation matrix
$P$ of compatible size, $\sigma(A)=\sigma(PAP^{\intercal})$.
\begin{lemma}
  \label{lem:permutation_invariance}Let $A=(a_{i_{1}\cdots i_{m}})$
  be a tensor, $\pi$ be a permutation of $\{1,\ldots,n\}$, and
  $A^{\prime}=(a_{i_{1}\cdots i_{m}}^{\prime})$
  be the tensor with entries
  \[
    a_{i_{1}\cdots i_{m}}^{\prime}=a_{\pi(i_{1})\cdots\pi(i_{m})}.
  \]
  Then, $\sigma(A)=\sigma(A^{\prime})$.
\end{lemma}
\begin{proof}
  Let $\lambda$ be an eigenvalue of $A$ with corresponding eigenvector
  $x$. Let $y$ denote the vector whose entries are $y_{i}=x_{\pi(i)}$.
  Then, for each $i$,
  \belowdisplayskip=-12pt
  \begin{align*}
    \lambda(y_{i})^{m-1}=\lambda(x_{\pi(i)})^{m-1} &
    =\sum_{i_{2},\ldots,i_{m}}a_{\pi(i)i_{2}\cdots
    i_{m}}x_{i_{2}}\cdots x_{i_{m}}                    \\
    &
    =\sum_{i_{2},\ldots,i_{m}}a_{\pi(i)\pi(i_{2})\cdots\pi(i_{m})}x_{\pi(i_{2})}\cdots
    x_{\pi(i_{m})} \\
    & =\sum_{i_{2},\ldots,i_{m}}a_{ii_{2}\cdots
    i_{m}}^{\prime}y_{i_{2}}\cdots y_{i_{m}}.
  \end{align*}\qedhere
\end{proof}
We are now ready to prove \cref{thm:wcdd}. We split the proof
into parts. In each part, we use $A=(a_{i_{1}\cdots i_{m}})$ to denote
a w.d.d. Z-tensor with nonnegative diagonals.
\begin{proof}[Proof of \ref{enu:1} implies \ref{enu:2}]
  This follows directly from \cref{prop:zhang_1}.
\end{proof}
\begin{proof}[Proof of \ref{enu:3} implies \ref{enu:1}]
  Suppose $A$ is w.c.d.d. Therefore, $A$ is w.d.d. by definition,
  and hence $A$ is an M-tensor by \cref{prop:zhang_2}. Let $\lambda$
  and $x$ be an eigenvalue-eigenvector pair of $A$. We may, without
  loss of generality, assume $\left\Vert x\right\Vert _{\infty}=1$
  (otherwise, define a new vector $y=x/\Vert x\Vert_{\infty}$ and note
  that $Ay^{m-1}=\lambda y^{[m-1]}$). Since $A$ is an M-tensor,
  $\operatorname{Re}\lambda\geq0$
  by \cref{prop:zhang_1}. In order to arrive at a contradiction,
  suppose $\operatorname{Re}\lambda=0$.

  Now, fix $i$ such that $|x_{i}|=\Vert x\Vert_{\infty}=\max_{j}|x_{j}|$.
  It follows that
  \[
    \left|\lambda-a_{i\cdots
    i}\right|\leq\sum_{(i_{2},\ldots,i_{m})\neq(i,\ldots,i)}\left|a_{ii_{2}\cdots
    i_{m}}\right|\left|x_{i_{2}}\right|\cdots\left|x_{i_{m}}\right|\leq\sum_{(i_{2},\ldots,i_{m})\neq(i,\ldots,i)}\left|a_{ii_{2}\cdots
    i_{m}}\right|.
  \]
  Since $\operatorname{Re}\lambda=0$ and $a_{i\cdots i}$ is real,
  \[
    \left|a_{i\cdots
    i}\right|=\left|\operatorname{Re}\lambda-a_{i\cdots
    i}\right|\leq\left|\lambda-a_{i\cdots i}\right|.
  \]
  Combining the above inequalities,
  \[
    \left|a_{i\cdots
    i}\right|\leq\sum_{(i_{2},\ldots,i_{m})\neq(i,\ldots,i)}\left|a_{ii_{2}\cdots
    i_{m}}\right|\left|x_{i_{2}}\right|\cdots\left|x_{i_{m}}\right|\leq\sum_{(i_{2},\ldots,i_{m})\neq(i,\ldots,i)}\left|a_{ii_{2}\cdots
    i_{m}}\right|.
  \]
  Since $A$ is w.d.d., the above chain of inequalities holds with equality
  so that $i\notin J(A)$ and $|x_{i_{2}}|=\cdots=|x_{i_{m}}|=1$ whenever
  $|a_{ii_{2}\cdots i_{m}}|\neq0$ for some $(i_{2},\ldots,i_{m})$.

  Since $A$ is w.c.d.d., we may pick a walk $i^{1}\rightarrow
  i^{2}\rightarrow\cdots\rightarrow i^{k}$
  starting at some row $i^{1}$ for which $|x_{i^{1}}|=1$ and ending
  at some row $i^{k}\in J(A)$. Setting $i=i^1$ in the previous
  paragraph, we get $|x_{i^{2}}|=1$ and therefore also $i^2\notin
  J(A)$. Applying this reasoning inductively, we conclude
  that $i^{k}\notin J(A)$, a contradiction.
\end{proof}
\begin{proof}[Proof of \ref{enu:2} implies \ref{enu:3}]
  Suppose $A$ is not w.c.d.d. Proceeding by contrapositive, it is
  sufficient to show that $\lambda=0$ is an eigenvalue of $A$. Let
  \[
    W(A)=\left\{ i^{1}\notin J(A)\colon\text{there exists a walk
      }i^{1}\rightarrow i^{2}\rightarrow\cdots\rightarrow i^{k}\text{
    such that }i^{k}\in J(A)\right\} .
  \]
  Let $R(A)=\{1,\ldots,n\}\setminus(J(A)\cup W(A))$. Since $A$ is
  not w.c.d.d., $R(A)$ is nonempty. By \cref{lem:permutation_invariance},
  we may assume that $R(A)=\{1,\ldots,r\}$ for some $1\leq r\leq n$
  (otherwise, permute the indices appropriately). For the remainder
  of the proof, let $e=(1,\ldots,1)^{\intercal}$ be the vector of ones
  in $\mathbb{R}^{r}$.

  If $r=n$, it follows that $J(A)$ is empty and hence every row of
  $A$ is not s.d.d. Since $A$ is a w.d.d. Z-tensor, we have, for each
  row $i$,
  \begin{equation}
    a_{i\cdots
    i}=-\sum_{(i_{2},\ldots,i_{m})\neq(i,\ldots,i)}a_{ii_{2}\cdots
    i_{m}}.\label{eq:singularity}
  \end{equation}
  In other words, $Ae^{m-1}=0$, and hence $\lambda=0$ is an eigenvalue
  of $A$.

  If $r<n$, the adjacency graph has the structure shown in \cref{fig:graph}.
  In particular, there are no edges from vertices $i\in R(A)$ to vertices
  $j\notin R(A)$ since if there were, $i$ would not be a member of
  $R(A)$ by definition. This implies that
  \[
    a_{ii_{2}\cdots i_{m}}=0\qquad\text{if }i\in R(A)\text{ and
    }i_{k}\notin R(A)\text{ for some }2\leq k\leq m.
  \]
  Equivalently,
  \begin{equation}
    a_{ii_{2}\cdots i_{m}}=0\qquad\text{if }i\leq r\text{ and
    }\max\{i_{2},\ldots,i_{m}\}>r.\label{eq:zero_block}
  \end{equation}
  Define the $m$-order and $n$-dimensional tensor $B=(b_{i_{1}\cdots i_{m}})$
  by
  \[
    b_{i_{1}\cdots i_{m}}=
    \begin{cases}
      a_{i_{1}\cdots i_{m}} & \text{if }1\leq i_{1},\ldots,i_{m}\leq r\\
      0 & \text{otherwise}.
    \end{cases}
  \]
  By \eqref{eq:zero_block}, it follows that $B$ is a w.d.d. Z-tensor
  with no s.d.d. rows and hence similarly to \eqref{eq:singularity},
  we can establish that for any vector $w$ in $\mathbb{R}^{n-r}$,
  \begin{equation}
    Bx^{m-1}=0\qquad\text{where}\qquad x=
    \begin{pmatrix}e\\
      w
    \end{pmatrix}.\label{eq:top_left}
  \end{equation}
  Define the $m$-order and $n$-dimensional tensors $C=(c_{i_{1}\cdots i_{m}})$
  and $D=(d_{i_{1}\cdots i_{m}})$ by
  \[
    c_{i_{1}\cdots i_{m}}=
    \begin{cases}
      1 & \text{if }1\leq i_{1}\leq r\text{ and
      }(i_{2},\ldots,i_{m})=(i_{1},\ldots,i_{1})\\
      0 & \text{otherwise}.
    \end{cases}
  \]
  and
  \[
    d_{i_{1}\cdots i_{m}}=
    \begin{cases}
      0 & \text{if }1\leq i_{1}\leq r\\
      a_{i_{1}\cdots i_{m}} & \text{otherwise}.
    \end{cases}
  \]
  By construction, $C+D$ is a w.c.d.d. Z-tensor with nonnegative diagonals.
  Since we have already proven \ref{enu:3} implies \ref{enu:1} in
  \cref{thm:wcdd}, we conclude that $C+D$ is a strong M-tensor.
  By \cref{cor:ding_corollary}, we can find a nonnegative vector $y$
  such that
  \begin{equation}
    (C+D)y^{m-1}=
    \begin{pmatrix}e\\
      0
    \end{pmatrix}.\label{eq:bottom_right}
  \end{equation}
  Note that if $1\leq i\leq r$, the above implies
  \[
    \sum_{i_{2}\cdots i_{m}}(c_{ii_{2}\cdots i_{m}}+d_{ii_{2}\cdots
    i_{m}})y_{i_{2}}\cdots y_{i_{m}}=y_{i}^{m-1}=1.
  \]
  Therefore,
  \[
    y=
    \begin{pmatrix}e\\
      w
    \end{pmatrix}
  \]
  for some vector $w$ in $\mathbb{R}^{n-r}$. Since $A=B+(C+D)-C$,
  \eqref{eq:top_left} and \eqref{eq:bottom_right} imply
  \[
    Ay^{m-1}=By^{m-1}+(C+D)y^{m-1}-Cy^{m-1}=0+
    \begin{pmatrix}e\\
      0
    \end{pmatrix}-
    \begin{pmatrix}e\\
      0
    \end{pmatrix}=0,
  \]
  so that $\lambda=0$ is an eigenvalue of $A$.
\end{proof}
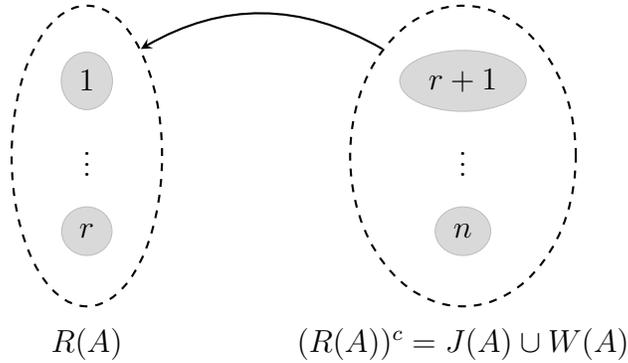
\begin{figure}[ht]
  \centering
  \begin{tikzpicture}
    \node [draw=black!25, ellipse, fill=black!15] (0) {$1$};
    \node [below of=0] (Rdots) {$\vdots$};
    \node [draw=black!25, ellipse, fill=black!15, below of=Rdots] (r) {$r$};
    \node [draw=black!25, ellipse, fill=black!15] at (5, 0) (r1) {$r+1$};
    \node [below of=r1] (SWdots) {$\vdots$};
    \node [draw=black!25, ellipse, fill=black!15, below of=SWdots] (M) {$n$};
    \node [ellipse, thick, dashed, draw, minimum width=2cm, minimum
    height=4cm] at (Rdots) (R) {};
    \node [ellipse, thick, dashed, draw, minimum width=3cm, minimum
    height=4cm] at (SWdots) (SW) {};
    \path [thick, ->, >=stealth] (SW.north west) edge [bend right]
    (R.north east);
    \node [below of=R, yshift=-1.5cm] {$R(A)$};
    \node [below of=SW, yshift=-1.5cm] {$(R(A))^c = J(A) \cup W(A)$};
  \end{tikzpicture}

  \caption{Adjacency graph in the proof of \cref{thm:wcdd}\label{fig:graph}}
\end{figure}

\section{\label{sec:bellman}High order Bellman equations}

We now return to the high order Bellman equation \eqref{eq:bellman},
repeated below for the reader's convenience:
\[
  \min_{P\in\mathcal{P}}\left\{ A(P)u^{m-1}-b(P)\right\} =0.
\]
In the above, $A(P)=(a_{i_{1}\cdots i_{m}}(P))$ is an $m$-order
and $n$-dimensional real tensor and $b(P)=(b_{1}(P),\ldots,b_{n}(P))$
is a vector in $\mathbb{R}^{n}$. It is understood that (cf. \cite{MR3493959})
\begin{enumerate}[label=(\emph{\roman*})]
  \item $\mathcal{P}=\mathcal{P}_{1}\times\cdots\times\mathcal{P}_{n}$ is
    a finite product of nonempty sets. That is, each $P=(P_{1},\ldots,P_{n})$
    in $\mathcal{P}$ is an $n$-tuple with $P_{i}$ in $\mathcal{P}_{i}$.
  \item \label{enu:decoupled}Policies are ``row-decoupled''. That is, for
    any two policies $P$ and $P^{\prime}$ in $\mathcal{P}$,
    $a_{ii_{2}\cdots i_{m}}(P)=a_{ii_{2}\cdots i_{m}}(P^{\prime})$
    and $b_{i}(P)=b_{i}(P^{\prime})$ whenever $P_{i}=P_{i}^{\prime}$.
    In other words, the $i$-th row of $A(P)$ and $b(P)$ are determined
    solely by $P_{i}$.
  \item \label{enu:coordinatewise}Infimums (and other extrema) are taken
    with respect to the coordinatewise order. That is, for
    $\{y^{(\alpha)}\}_{\alpha}\subset\mathbb{R}^{n}$,
    $\inf_{\alpha}y^{(\alpha)}$ is a vector whose $i$-th entry is
    $\inf_{\alpha}y_{i}^{(\alpha)}$.
    For example, $\min\{(
        \begin{smallmatrix}1\\
          0
      \end{smallmatrix}),(
        \begin{smallmatrix}0\\
          1
    \end{smallmatrix})\}=(
      \begin{smallmatrix}0\\
        0
    \end{smallmatrix})$.
\end{enumerate}
We require the following assumptions to study the problem:
\begin{enumerate}[label=(H\arabic*)]
  \item \label{enu:positive}$b(P)$ is a positive vector for each $P$ in
    $\mathcal{P}$.
  \item \label{enu:continuous}$\mathcal{P}$ is a compact topological space
    and $A:\mathcal{P}\rightarrow\mathbb{R}^{n^{m}}$ and
    $b:\mathcal{P}\rightarrow\mathbb{R}^{n}$
    are continuous functions.
\end{enumerate}
In practice, $\mathcal{P}$ is usually finite
(\cref{rem:finite_disclaimer}) in which case \ref{enu:continuous} is
trivially satisfied.

\subsection{Existence and uniqueness}

We now establish existence and uniqueness of positive solutions to
\eqref{eq:bellman}.
\begin{lemma}[Uniqueness]
  \label{lem:uniqueness}Suppose \ref{enu:positive}, \ref{enu:continuous},
  and $A(P)$ is a strong M-tensor for each $P$ in $\mathcal{P}$.
  Then, there is at most one positive solution $u$ of \eqref{eq:bellman}.
\end{lemma}
\begin{proof}
  Let $u$ and $w$ be two positive solutions of \eqref{eq:bellman}.
  By the compactness of $\mathcal{P}$ and continuity of $A$ and $b$,
  we can find $P^{*}$ such that
  \[
    A(P^{*})u^{m-1}-b(P^{*})=0=\min_{P\in\mathcal{P}}\left\{
    A(P)w^{m-1}-b(P)\right\} \leq A(P^{*})w^{m-1}-b(P^{*}).
  \]
  Therefore,
  \[
    0<A(P^{*})u^{m-1}\leq A(P^{*})w^{m-1}.
  \]
  Using the fact that $(A(P^{*}))_{++}^{-1}$ is nondecreasing, applying
  the function $(A(P^{*}))_{++}^{-1}$ to the above inequality yields
  $u\leq w$. Reversing the roles of $u$ and $w$, we obtain the reverse
  inequality.
\end{proof}
\begin{lemma}[Existence I]
  \label{lem:existence}Suppose \ref{enu:positive}, \ref{enu:continuous},
  and $A(P)$ is an s.d.d. M-tensor for each $P$ in $\mathcal{P}$.
  Then, there exists a positive solution $u$ of \eqref{eq:bellman}.
\end{lemma}
A close examination of the proof below reveals that we can relax the
requirement that ``$b(P)$ is positive'' in \ref{enu:positive}
to ``$b(P)$ is nonnegative''. In this case, the arguments establish
the existence of a nonnegative solution $u$.
\begin{proof}
  We claim that it is sufficient to consider the case in which $1 -
  a_{i\cdots i}(P)=0$ for all $i$ (we will come back to this claim later).
  Note that $u$ is a solution of \eqref{eq:bellman} if and only if it
  is a fixed point of the map $F$ defined by
  \[
    F(u)=\max_{P\in\mathcal{P}}\left\{ (I-A(P))u^{m-1}+b(P)\right\}
    ^{[1/(m-1)]}.
  \]
  Since the diagonals of $I-A(P)$ are zero, the off-diagonals of
  $A(P)$ are nonpositive, and $b(P)$ is positive, it follows that $F$
  maps nonnegative vectors to positive vectors (i.e.,
  $F(\mathbb{R}_{+}^{n})\subset\mathbb{R}_{++}^{n}$).
  Next,
  we prove that $F$ is continuous on $\mathbb{R}_{+}^{n}$. In order
  to do so, it is sufficient to show that the function $G$ defined
  by $G(u)=(F(u))^{[m-1]}$ is locally Lipschitz on $\mathbb{R}_{+}^{n}$.
  Indeed, for nonnegative vectors $u$ and $w$,
  \begin{multline*}
    \left\Vert G(u)-G(w)\right\Vert
    _{\infty}\leq\max_{P\in\mathcal{P}}\left\Vert
    (I-A(P))u^{m-1}-\left(I-A(P)\right)w^{m-1} \right\Vert_{\infty}\\
    \leq\max_{P\in\mathcal{P}}\max_{i}\sum_{(i_{2},\ldots,i_{m})\neq(i,\ldots,i)}\left|a_{ii_{2}\cdots
    i_{m}}(P)\right|\max\left\{ \left\Vert u\right\Vert
    _{\infty}^{m-2},\left\Vert w\right\Vert _{\infty}^{m-2}\right\}
    \sum_{j=2}^{m}\left|u_{i_{j}}-w_{i_{j}}\right|\\
    \leq\operatorname{const.}\max\left\{ \left\Vert u\right\Vert
    _{\infty}^{m-2},\left\Vert w\right\Vert _{\infty}^{m-2}\right\}
    \left\Vert u-w\right\Vert _{\infty}
  \end{multline*}
  where $\operatorname{const.}$ is a positive constant which does not
  depend on $u$ or $w$. Note that in the $m=2$ case, $\Vert
  u\Vert_{\infty}^{m-2}=\Vert w\Vert_{\infty}^{m-2}=1$,
  and hence this argument establishes \emph{global} Lipschitzness. Next,
  we derive some bounds on $F$. The triangle inequality yields
  \[
    \left\Vert F(u)\right\Vert
    _{\infty}\leq\max_{P\in\mathcal{P}}\left\{ \left\Vert
      u\right\Vert
      _{\infty}^{m-1}\max_{i}\sum_{(i_{2},\ldots,i_{m})\neq(i,\ldots,i)}\left|a_{ii_{2}\cdots
    i_{m}}(P)\right|+\left\Vert b(P)\right\Vert _{\infty}\right\} ^{1/(m-1)}.
  \]
  Therefore, there exist $i^{*}$ and $P^{*}$ such that
  \[
    \left\Vert F(u)\right\Vert _{\infty}\leq\left\Vert u\right\Vert
    _{\infty}\left(\sum_{(i_{2},\ldots,i_{m})\neq(i^{*},\ldots,i^{*})}\left|a_{i^{*}i_{2}\cdots
    i_{m}}(P^{*})\right|\right)^{1/(m-1)}+\left\Vert
    b(P^{*})\right\Vert _{\infty}^{1/(m-1)}.
  \]
  Since $A(P^{*})$ is s.d.d.,
  \[
    \sum_{(i_{2},\ldots,i_{m})\neq(i^{*},\ldots,i^{*})}\left|a_{i^{*}i_{2}\cdots
    i_{m}}(P^{*})\right|<a_{i^{*}\cdots i^{*}}(P^{*})=1.
  \]
  In other words, there exist positive constants $C_{1}<1$ and $C_{2}$
  such that for any $u$,
  \[
    \left\Vert F(u)\right\Vert _{\infty}\leq C_{1}\left\Vert
    u\right\Vert _{\infty}+C_{2}.
  \]
  Now, let $M=C_{2}/(1-C_{1})$ and
  $K=\{u\in\mathbb{R}_{+}^{n}\colon\Vert u\Vert_{\infty}\leq M\}$.
  By the above, $F(K)\subset K$. By the Schauder fixed point theorem,
  $F$ admits a fixed point in $K$. Moreover, since $F(K)\subset
  F(\mathbb{R}_{+}^{n})\subset\mathbb{R}_{++}^{n}$,
  this fixed point must be positive.

  Now, let us return to the unproven claim in the previous paragraph.
  Let
  \[
    D(P)=\operatorname{diag}(a_{1\cdots1}(P),\ldots,a_{n\cdots n}(P))
  \]
  be the diagonal matrix obtained from the diagonal entries of $A(P)$.
  Note, in particular, that the $i$-th entry of the vector
  $D(P)^{-1}(A(P)u^{m-1})$
  is
  \[
    \sum_{i_{2},\ldots,i_{m}}\frac{a_{ii_{2}\cdots
    i_{m}}(P)}{a_{i\cdots i}(P)}u_{i_{2}}\cdots u_{i_{m}}.
  \]
  Therefore, to establish the claim, it is sufficient to show that if
  $u$ satisfies
  \begin{equation}
    \min_{P\in\mathcal{P}}\left\{
    D(P)^{-1}\left(A(P)u^{m-1}-b(P)\right)\right\} =0,\label{eq:scaled}
  \end{equation}
  then $u$ is a solution of \eqref{eq:bellman} (while the converse
  is also true, we do not require it).
  Indeed, if $u$ satisfies \eqref{eq:scaled}, then
  $D(P^{*})^{-1}(A(P^{*})u^{m-1}-b(P^{*}))=0$ for some $P^{*}$.
  Multiplying both sides of this equation by $D(P^{*})$, we get
  $A(P^{*})u^{m-1}-b(P^{*})=0$ and hence
  \[
    \min_{P\in\mathcal{P}}\left\{A(P)u^{m-1}-b(P)\right\}\leq0.
  \]
  To establish the reverse inequality, we proceed by contradiction,
  assuming that we can find $P_{*}$ such that the vector
  $A(P_{*})u^{m-1}-b(P_{*})$
  has a strictly negative entry. In this case,
  $D(P_{*})^{-1}(A(P_{*})u^{m-1}-b(P_{*}))$
  also has a strictly negative entry, contradicting \eqref{eq:scaled}.
  Therefore, $u$ is a solution of \eqref{eq:bellman}, as desired.
\end{proof}
We would like to extend the above existence result to w.c.d.d. M-tensors.
In order to do so, we require the following intermediate result, which
is of independent interest.
\begin{lemma}
  \label{lem:bounded}Let $A$ be a strong M-tensor and $b$ be a positive
  vector (of compatible size). Then, the set
  \[
    \left\{ (A+\epsilon I)_{++}^{-1}(b)\colon\epsilon\geq0\right\}
  \]
  is bounded.
\end{lemma}
\begin{proof}
  Since $A$ is a strong M-tensor there exists a positive vector $z$
  such that $Az^{m-1}$ is positive \cite[Theorem 3]{MR3116429}.
  This in turn implies that
  \[
    (A+\epsilon I)z^{m-1}=Az^{m-1}+\epsilon z^{[m-1]}\geq Az^{m-1}>0.
  \]
  Now, defining
  \[
    \overline{\gamma}=\max_{i}\frac{b_{i}}{(\left(A+\epsilon
    I\right)z^{m-1})_{i}}\leq\max_{i}\frac{b_{i}}{(Az^{m-1})_{i}},
  \]
  the arguments in the proof of \cite[Theorem 3.2]{MR3519197} imply
  that the unique positive solution $x$ of $(A+\epsilon I)x^{m-1}=b$
  satisfies
  \[
    x\leq\overline{\gamma}^{1/(m-1)}z,
  \]
  giving us an upper bound that is independent of $\epsilon$.
\end{proof}
\begin{lemma}[Existence II]
  \label{lem:existence_wcdd}Suppose \ref{enu:positive}, $\mathcal{P}$
  is finite, and $A(P)$ is a w.c.d.d. M-tensor for each $P$ in $\mathcal{P}$.
  Then, there exists a positive solution $u$ of \eqref{eq:bellman}.
\end{lemma}
\begin{proof}
  As in the proof of \cref{lem:existence}, it is sufficient to
  consider the case in which $a_{i\cdots i}(P)=1$. Now, let $k$ be
  a positive integer. Since $A(P)$ is w.d.d., it follows that $A(P)+k^{-1}I$
  is s.d.d. Therefore, by \cref{lem:existence}, we can find a
  positive vector $u_{(k)}$ and a policy $P^{k}$ such that
  \[
    \min_{P\in\mathcal{P}}\left\{
    \left(A(P)+k^{-1}I\right)u_{(k)}^{m-1}-b(P)\right\}
    =\left(A(P^{k})+k^{-1}I\right)u_{(k)}^{m-1}-b(P^{k})=0.
  \]

  Since the sequence $(P^{k})_{k}$ has finite range (due to the finitude
  of $\mathcal{P}$), the pigeonhole principle affords us the existence
  of an increasing sequence $(k_{\ell})_{\ell}$ of positive integers
  and a policy $P^{*}$ such that for all $\ell$,
  \begin{equation}
    \min_{P\in\mathcal{P}}\left\{
    \left(A(P)+k_{\ell}^{-1}I\right)u_{(k_{\ell})}^{m-1}-b(P)\right\}
    =\left(A(P^{*})+k_{\ell}^{-1}I\right)u_{(k_{\ell})}^{m-1}-b(P^{*})=0.\label{eq:subsequence}
  \end{equation}
  For brevity, let $A=A(P^{*})$ and $b=b(P^{*})$. Since
  \[
    u_{(k_{\ell})}=(A+k_{\ell}^{-1}I)_{++}^{-1}(b),
  \]
  \cref{lem:bounded} implies that the sequence $(u_{(k_{\ell})})_{\ell}$
  is contained in a compact set and thereby admits a convergent subsequence
  with limit $u_{(\infty)}$.

  Now, we show that $u_{(\infty)}$ is a solution of \eqref{eq:bellman}.
  First, note that $k_{\ell}^{-1}Iu_{(k_{\ell})}^{m-1}\rightarrow0$
  as $\ell\rightarrow\infty$. Therefore, it is sufficient to establish
  that the function $H$ defined by
  \[
    H(u)=\min_{P\in\mathcal{P}}\left\{ A(P)u^{m-1}-b(P)\right\}
  \]
  is continuous on $\mathbb{R}_{+}^{n}$ and take limits in
  \eqref{eq:subsequence}
  to arrive at the desired result. This follows immediately from the
  fact that $H(u)=u^{[m-1]}-G(u)$ where $G$ is the locally Lipschitz
  (and hence continuous) function defined in the proof of \cref{lem:existence}.

  Note that $u_{(\infty)}$ is positive: if $(u_{(\infty)})_r=0$ for
  some $r$, then, since each $A(P)$ is a Z-tensor and
  $u_{(\infty)}\geq0$, we have
  $(A(P)u_{(\infty)}^{m-1})_r\leq0<b_r(P)$ for all $P$, contradicting
  \eqref{eq:bellman}.
\end{proof}

\begin{rem}
  The proof of \cref{lem:existence_wcdd} uses a pigeonhole principle
  which relies on the assumed finitude of the policy set $\mathcal{P}$.
  Whether this assumption can be dropped remains an interesting open question.

  From a practical perspective, it is important to note that
  classical $m=2$ Bellman equations appear almost exclusively with
  finite policy sets $\mathcal{P}$ in applications.
  For example, in the context of discretizations of differential
  equations from optimal control, the policy set is always chosen to
  be a discretization of the corresponding control set so that the
  problem can be made amenable to numerical computation (see, e.g.,
  \cite[Section 5.2]{MR2551155}).
  Analogously, we do not expect the finiteness assumption to be
  particularly obstructive in the $m > 2$ case.
  Indeed, the applications studied in \cref{sec:application} involve
  finite policy sets.
  \label{rem:finite_disclaimer}
\end{rem}

\subsection{Policy iteration}

In the classical $m=2$ setting, a popular computational procedure
to solve \eqref{eq:bellman} is \emph{policy iteration}. We give a
brief sketch of the algorithm and refer to \cite{MR2551155} for details.
At the $k$-th iteration, the algorithm picks a policy $P^{k}$ in
$\mathcal{P}$ and solves the system $A(P^{k})u_{(k)}=b(P^{k})$.
The policy $P^{k}$ is picked to ensure $u_{(k-1)}\leq u_{(k)}$ so
that $u=\lim_{k}u_{(k)}$
exists.
Using continuity arguments, it can be shown that this limit is a
solution of \eqref{eq:bellman}.

When $\mathcal{P}$ is finite, policy iteration takes at most $|\mathcal{P}|$
iterations before achieving the limit (i.e.,
$u_{(|\mathcal{P}|)}=u_{(|\mathcal{P}|+1)}=\cdots=u$).
Analogously to the simplex algorithm, whose worst case complexity
is determined by the number of vertices in the feasible polytope,
policy iteration generally terminates in far fewer iterations. Continuing
our analogy, we call the map which associates to each iteration $k$ a
policy $P^{k}$ a \emph{pivot rule}.

Below, we present an obvious extension of policy iteration to the case of $m>2$.
In the statement of the algorithm, we allow for some
freedom in the choice of pivot rule.
Unlike the $m=2$ case, it is not clear if there exists a pivot rule
which ensures $u_{(k-1)}\leq u_{(k)}$.
The resulting algorithm is below.

\begin{algorithm}[H]
  \vspace{6pt}

  {\algorithmcodesize
    \begin{algorithmic}[1]
      \Procedure{Policy-Iteration}{$A$, $b$}
      \For{$k \gets 1,\ldots,|\mathcal{P}|$}
      \State Pick $P^k$ in $\mathcal{P} \setminus \{ P^1, \ldots,
      P^{k-1} \}$ according to some pivot rule \label{line:pivoting}
      \label{line:policy_improvement}
      \State Solve the tensor equation $A(P^k) u_{(k)}^{m-1} =
      b(P^k)$ for $u_{(k)}$ in $\mathbb{R}_{++}^n$
      \label{line:policy_evaluation}
      \If{$\min_{P \in \mathcal{P}} \{ A(P) u_{(k)}^{m-1} - b(P) \} =
      0$} \label{line:equality}
      \State \Return $u_{(k)}$
      \EndIf
      \EndFor
      \State \textbf{error} no solution found
      \EndProcedure
    \end{algorithmic}
  }

  \vspace{6pt}

  \caption{A policy iteration algorithm for
  \eqref{eq:bellman}\label{alg:policy_iteration}}
\end{algorithm}

\begin{rem}
  The terminating condition on line \ref{line:equality}, while
  convenient for a theoretical discussion, is unsuitable for a
  practical implementation.
  Such an implementation should use instead a condition on the
  relative error between iterates $u_{(k-1)}$ and $u_{(k)}$ (see,
  e.g., \eqref{eq:relative_error}) or a condition on the norm of
  $\min_{P \in \mathcal{P}} \{ A(P) u_{(k)}^{m-1} - b(P) \}$.
\end{rem}

\begin{theorem}
  \label{thm:bellman}Suppose \ref{enu:positive}, $\mathcal{P}$ is
  finite, and $A(P)$ is a w.c.d.d. M-tensor for each $P$ in $\mathcal{P}$.
  Then, {\sc Policy-Iteration} returns the unique positive solution
  of \eqref{eq:bellman}.
\end{theorem}
As usual, we can relax the requirement that ``$b(P)$ is positive'' in
\ref{enu:positive} to ``$b(P)$ is nonnegative'' by replacing
$\mathbb{R}_{++}^n$ with $\mathbb{R}_+^n$ on line
\ref{line:policy_evaluation} of the algorithm, with the understanding
that the tensor equation is solved using $(A(P^k))_+^{-1}$ from
\cref{cor:ding_corollary}.
In this case, the algorithm returns a nonnegative solution $u$, which
may not be unique.
\begin{proof}
  This is a direct consequence of \cref{lem:uniqueness} and
  \cref{lem:existence_wcdd} along with the fact that the algorithm
  iterates over all policies.
\end{proof}

Taking $w = u_{(k)}$ and $P^\prime = P^k$ in the result below
establishes that the solution $u$ described in \cref{thm:bellman}
dominates the iterates $u_{(k)}$ generated by the algorithm.

\begin{lemma}
  Suppose \ref{enu:positive}, \ref{enu:continuous}, and $A(P)$ is a
  strong M-tensor for each $P$ in $\mathcal{P}$.
  If $u$ is a positive solution of \eqref{eq:bellman} and $w$ is a
  positive vector satisfying $A(P^{\prime})w^{m-1}=b(P^{\prime})$ for
  some $P^\prime$ in $\mathcal{P}$, then $u\geq w$.
\end{lemma}

\begin{proof}
  Since
  \[
    A(P^{\prime})u^{m-1}-b(P^{\prime})\geq\min_{P\in\mathcal{P}}\left\{
    A(P)u^{m-1}-b(P)\right\} =0=A(P^{\prime})w^{m-1}-b(P^{\prime}),
  \]
  it follows that $A(P^{\prime})u^{m-1}\geq A(P^{\prime})w^{m-1}>0$.
  The desired result follows by applying the function
  $(A(P^{\prime}))_{++}^{-1}$ to this inequality.
\end{proof}

\subsection{Locally optimal policy}

The pivot rule which we employ in the numerical tests appearing in
the sequel is inspired by the classical ($m=2$) policy iteration
algorithm. The idea behind the pivot rule is simple: at the $k$-th
iteration, let
\[
  \mathcal{O}=\underset{P\in\mathcal{P}}{\operatorname{arg\,min}}\left\{
  A(P)u_{(k-1)}^{m-1}-b(P)\right\}
\]
be the set of policies that are ``locally optimal'' for $u_{(k-1)}$
where, for convenience, we define $u_{(0)}=0$. Let
\[
  \mathcal{H}=\mathcal{P}\setminus\{P^{1},\ldots,P^{k-1}\}
\]
be the set of policies the algorithm has not yet considered. If
$\mathcal{O}\cap\mathcal{H}\neq\emptyset$,
we pick $P^{k}$ in $\mathcal{O}\cap\mathcal{H}$. Otherwise, we pick
$P^{k}$ in $\mathcal{H}$. It is readily verified that if $m=2$,
we retrieve the classical policy iteration algorithm \cite[Algorithm
Ho-1]{MR2551155}.

\subsection{Incorporating lower order tensors}

Some of the results of the previous sections can also be applied to
the following
more general higher order Bellman equation
\begin{equation}
  \min_{P\in\mathcal{P}}\left\{
  A(P)u^{m-1}-B_{m-1}(P)u^{m-2}-B_{m-2}(P)u^{m-3}-\cdots-B_{2}(P)u-b(P)\right\}
  =0\label{eq:general}
\end{equation}
where each $B_{p}(P)$ is a row-decoupled (see \ref{enu:decoupled}
at the beginning of \cref{sec:bellman}) $p$-order $n$-dimensional
nonnegative tensor. This is possible since if $A(P)$ is an $m$-order
$n$-dimensional strong M-tensor for each $P$, then $u$ is a solution
of \eqref{eq:general} if and only if $w=
\begin{pmatrix}u & 1
\end{pmatrix}^{\intercal}$
is a solution of
\[
  \min_{P\in\mathcal{P}}\left\{ A^{\prime}(P)w^{m-1}-
    \begin{pmatrix}b(P)\\
      1
  \end{pmatrix}\right\} =0
\]
where for each $P$, $A^{\prime}(P)$ is an appropriately chosen $m$-order
($n+1$)-dimensional strong M-tensor whose construction is detailed
in the proof of the next lemma.
\begin{lemma}
  \label{lem:transform}Let $A=(a_{i_{1}\cdots i_{m}})$ be an $m$-order
  $n$-dimensional strong M-tensor and $B_{p}=(b_{i_{1}\cdots i_{p}})$
  be a $p$-order $n$-dimensional nonnegative tensor for $2\leq p<m$.
  Then, there exists an $m$-order ($n+1$)-dimensional strong M-tensor
  $A^{\prime}$ such that
  \begin{equation}
    \begin{pmatrix}Ax^{m-1}-B_{m-1}x^{m-2}-B_{m-2}x^{m-3}-\cdots-B_{2}x\\
      1
    \end{pmatrix}=A^{\prime}
    \begin{pmatrix}x\\
      1
    \end{pmatrix}^{m-1}\text{ for all }x\in\mathbb{R}^{n}.\label{eq:equality}
  \end{equation}
\end{lemma}
\begin{proof}
  We claim that we can construct an $m$-order ($n+1$)-dimensional
  Z-tensor $A^{\prime}=(a_{i_{1}\cdots i_{m}}^{\prime})$ satisfying
  \eqref{eq:equality}. Indeed, if this is the case, since $A$ is a
  strong M-tensor, we can find a positive vector $v$ such that
  \[
    Av^{m-1}-B_{m-1}v^{m-2}-B_{m-2}v^{m-3}-\cdots-B_{2}v>0
  \]
  (see the proof of \cite[Theorem 3.6]{MR3519197} for details) and
  hence
  \[
    A^{\prime}
    \begin{pmatrix}v\\
      1
    \end{pmatrix}^{m-1}>0.
  \]
  Therefore, $A^{\prime}$ is semi-positive and hence a strong M-tensor
  \cite[Theorem 3]{MR3116429}.

  Returning to the claim above, we give the construction in the case
  of $m=3$, from which the general $m\geq3$ case should be evident.
  Indeed, in the case of $m=3$, we can take the nonzero entries of
  $A^{\prime}$ to be
  \begin{align*}
    a_{i,j,k}^{\prime} & =a_{i,j,k}, & \text{if }1\leq i,j,k\leq n\\
    a_{i,j,n+1}^{\prime}=a_{i,n+1,j}^{\prime} & =-b_{i,j}/2, &
    \text{if }1\leq i,j\leq n\\
    a_{n+1,n+1,n+1}^{\prime} & =1.
  \end{align*}
  Clearly, $A^{\prime}$ is a Z-tensor. Now, let $x$ in $\mathbb{R}^{n}$
  be arbitrary and $y=
  \begin{pmatrix}x & 1
  \end{pmatrix}^{\intercal}$.
  Then, for $1\leq i\leq n$,
  \begin{multline*}
    (Ax^{2}-Bx)_{i}=\sum_{1\leq j,k\leq
    n}a_{i,j,k}x_{j}x_{k}-\sum_{1\leq j\leq n}b_{i,j}x_{j}\\
    =\sum_{1\leq j,k\leq
    n}a_{i,j,k}^{\prime}\left(x_{j}x_{k}\right)+\sum_{1\leq j\leq
    n}a_{i,j,n+1}^{\prime}\left(x_{j}\cdot1\right)+\sum_{1\leq j\leq
    n}a_{i,n+1,j}^{\prime}\left(1\cdot x_{j}\right)\\
    =\sum_{1\leq j,k\leq n+1}a_{i,j,k}^{\prime}y_{j}y_{k}=(A^{\prime}y^{2})_{i}
  \end{multline*}
  so that \eqref{eq:equality} is satisfied.
\end{proof}

\section{\label{sec:application}Application to optimal stochastic control}

In this section, we apply our results to solve numerically the differential
equation
\begin{equation}
  \begin{cases}
    -{\displaystyle
    \max_{(\gamma,\lambda)\in(\Gamma,\Lambda)}}\left\{
      L^{\lambda}U-\eta
    U-\frac{1}{2}\alpha\gamma^{2}U+\beta\gamma\right\} =0, & \text{on }\Omega\\
    U=g, & \text{on }\partial\Omega
  \end{cases}\label{eq:pde}
\end{equation}
where $L^{\lambda}$ is the (possibly degenerate) elliptic operator
\[
  L^{\lambda}U(x)=\frac{1}{2}\sigma(x,\lambda)^{2}U^{\prime\prime}(x)+\mu(x,\lambda)U^{\prime}(x).
\]
and
\[
  \Omega=(0,1)\text{,}\qquad\Gamma=[0,\infty)\text{,}\qquad\text{and}\qquad\Lambda\text{
  is a compact metric space}.
\]
We require the following assumptions:
\begin{enumerate}[label=(A\arabic*)]
  \item \label{enu:sets}Letting $d_{H}$ denote the Hausdorff metric,
    to each $\Delta x>0$, we can associate a finite
    subset $\Lambda_{\Delta x}$ of $\Lambda$ such that
    $d_{H}(\Lambda_{\Delta x},\Lambda)\rightarrow0$
    as $\Delta x\downarrow0$.
  \item \label{enu:functions}$\sigma,\mu$ (resp. $\eta,\alpha,\beta,g$)
    are real (resp. positive) continuous maps with
    $\operatorname{dom}(\sigma)=\operatorname{dom}(\mu)=\overline{\Omega}\times\Lambda$
    (resp.
    $\operatorname{dom}(\eta)=\operatorname{dom}(\alpha)=\operatorname{dom}(\beta)=\operatorname{dom}(g)=\overline{\Omega}$).
\end{enumerate}
Note that \ref{enu:sets} simply says that we can approximate $\Lambda$
by finite subsets.

Now, there are two ways to discretize \eqref{eq:pde}: a ``discretize
then optimize'' (DO) approach and an ``optimize then discretize''
(OD) approach. In the DO approach, we first replace the unbounded
control set $\Gamma$ by a partition of the interval
$\Gamma_{0}=[0,\gamma_{\max}]$
(for some $\gamma_{\max}>0$ chosen large enough). Next, we replace
the various quantities $U_{xx}$, $U_{x}$, and $U$ by their discrete
approximations. The resulting system is a classical $m=2$ Bellman
equation, which can be solved by policy iteration. Since the DO approach
is well-understood, we present its derivation in \cref{app:do}.

In the OD approach, we first find the point $\gamma$ at which the
maximum
\begin{equation}
  \max_{\gamma\in\Gamma}\left\{
  -\frac{1}{2}\alpha\gamma^{2}U+\beta\gamma\right\} \label{eq:optimization}
\end{equation}
is attained. Substituting this back into \eqref{eq:pde}, we discretize
the resulting differential equation. The OD approach results in a
scheme with lower truncation error.

In general, applying an OD approach to an elliptic differential equation
may result in a scheme which is nonmonotone and/or hard to solve (see
the discussion in \cite{forsyth2007numerical}). This is problematic,
since it is well-known that nonmonotone schemes are not guaranteed
to converge \cite{MR2218974}. In our case, the resulting OD system
ends up being a higher order Bellman equation involving a w.c.d.d.
M-tensor, making it both monotone and easy to solve by policy iteration
(recall \cref{thm:bellman}).
\begin{rem}[Connection to optimal stochastic control]
  Let $W=(W_{t})_{t\geq0}$ be a standard Brownian motion on a filtered
  probability space satisfying the usual conditions. It is well-known
  that (under some mild conditions), the value function
  \begin{multline*}
    v(x)=\sup_{\gamma,\lambda}\mathbb{E}\biggl[\int_{0}^{\tau}\exp\left(\int_{0}^{t}-\eta(X_{s}^{\lambda})-\frac{1}{2}\alpha(X_{s}^{\lambda})\gamma_{s}^{2}ds\right)\beta(X_{t}^{\lambda})\gamma_{t}dt\\
    +\exp\left(\int_{0}^{\tau}-\eta(X_{s}^{\lambda})-\frac{1}{2}\alpha(X_{s}^{\lambda})\gamma_{s}^{2}ds\right)g(X_{\tau}^{\lambda})\biggr]
  \end{multline*}
  is a viscosity solution of \eqref{eq:pde} \cite{MR2976505}. In the
  above, the supremum is over all progressively measurable processes
  $\gamma=(\gamma_{t})_{t\geq0}$ and $\lambda=(\lambda_{t})_{t\geq0}$
  taking values in $\Gamma$ and $\Lambda$, respectively, and
  \[
    X_{t}^{\lambda}=x+\int_{0}^{t}\mu(X_{s}^{\lambda},\lambda_{s})ds+\int_{0}^{t}\sigma(X_{s}^{\lambda},\lambda_{s})dW_{s}\qquad\text{and}\qquad\tau=\inf\{t\geq0\colon
    X_{t}^{\lambda}\notin\Omega\}.
  \]
  To ensure that the process $X^{\lambda}$ is well-defined, one should
  impose some additional assumptions (e.g., $\sigma(\cdot,\lambda)$
  and $\mu(\cdot,\lambda)$ are Lipschitz uniformly in $\lambda$).
\end{rem}

\subsection{\label{subsec:od}Optimize then discretize scheme}

In this subsection, we derive the OD scheme and prove that it
converges to the solution of \eqref{eq:pde}.
Let $\Delta x=1/M$ for some positive integer $M$. We write
$u_{i}\approx U(i\Delta x)$
for the numerical solution at $i\Delta x$ and let $u=(u_{0},\ldots,u_{M})$.
We denote by
\[
  (L_{\Delta
  x}^{\lambda}u)_{i}=\frac{1}{2}\sigma_{i}(\lambda)^{2}\frac{u_{i-1}-2u_{i}+u_{i+1}}{(\Delta
  x)^{2}}+\mu_{i}(\lambda)\frac{1}{\Delta x}
  \begin{cases}
    u_{i+1}-u_{i}, & \text{if }\mu_{i}(\lambda)\geq0\\
    u_{i}-u_{i-1}, & \text{if }\mu_{i}(\lambda)<0
  \end{cases}
\]
a standard upwind discretization of $L^{\lambda}U(i\Delta x)$ where,
for brevity, we have defined $\sigma_{i}(\lambda)=\sigma(i\Delta x,\lambda)$
and $\mu_{i}(\lambda)=\mu(i\Delta x,\lambda)$.

First, note that a solution $U$ of \eqref{eq:pde} must be everywhere
positive since otherwise \eqref{eq:optimization} is unbounded. By
virtue of this, the maximum in \eqref{eq:optimization} is
\[
  \max_{\gamma}\left\{
  -\frac{1}{2}\alpha\gamma^{2}U+\beta\gamma\right\} ={\displaystyle
  \frac{1}{2}\frac{\beta^{2}}{\alpha}\frac{1}{U}}.
\]
This suggests approximating \eqref{eq:pde} by the $M+1$ ``discrete''
equations
\begin{equation}
  \begin{cases}
    -{\displaystyle \max_{\lambda\in\Lambda_{\Delta x}}}\left\{
      (L_{\Delta x}^{\lambda}u)_{i}-\eta_{i}u_{i}+{\displaystyle
    \frac{1}{2}\frac{\beta_{i}^{2}}{\alpha_{i}}\frac{1}{u_{i}}}\right\}
    =0, & \text{for }0<i<M\\
    u_{i}=g_{i}, & \text{for }i=0,M
  \end{cases}\label{eq:od}
\end{equation}
where we have defined $\alpha_{i}=\alpha(i\Delta x)$,
$\beta_{i}=\beta(i\Delta x)$,
$\eta_{i}=\eta(i\Delta x)$, and $g_{i}=g(i\Delta x)$.

The difficulty in the OD approach is that \eqref{eq:od} cannot be
written as a classical $m=2$ Bellman equation due to the term $1/u_{i}$.
We resolve this by writing \eqref{eq:od} as a Bellman equation of
order $m=3$ instead. In order to do so, we first note that a positive
vector $u=(u_{0},\ldots,u_{M})$ satisfies \eqref{eq:od} if and only
if it satisfies
\begin{equation}
  \begin{cases}
    -{\displaystyle \max_{\lambda\in\Lambda_{\Delta x}}}\left\{
      u_{i}(L_{\Delta
      x}^{\lambda}u)_{i}-\eta_{i}u_{i}^{2}+{\displaystyle
    \frac{1}{2}\frac{\beta_{i}^{2}}{\alpha_{i}}}\right\} =0, &
    \text{for }0<i<M\\
    u_{i}^{2}=g_{i}^{2}, & \text{for }i=0,M.
  \end{cases}\label{eq:discrete_equations}
\end{equation}
To see this, multiply each equation in \eqref{eq:od} by $u_{i}$
(conversely, divide each equation in \eqref{eq:discrete_equations}
by $u_{i}$).
Next, define the Cartesian product
$\boldsymbol{\Lambda}_{\Delta x} = (\Lambda_{\Delta x})^{M+1}$ and
denote by $\boldsymbol{\lambda}=(\lambda_{0},\ldots,\lambda_{M})$
an element of $\boldsymbol{\Lambda}_{\Delta x}$ with
$\lambda_{i}\in\Lambda_{\Delta x}$.
Define $A(\boldsymbol{\lambda})=(a_{ijk}(\boldsymbol{\lambda}))$
as the order $m=3$ tensor whose only nonzero entries are
\stepcounter{equation}
\begin{align}
  a_{i,i-1,i}(\boldsymbol{\lambda})
  & =a_{i,i,i-1}(\boldsymbol{\lambda}),
  \nonumber \\
  2a_{i,i,i-1}(\boldsymbol{\lambda})
  & =-\frac{1}{2}\sigma_{i}(\lambda_{i})^{2}\frac{1}{(\Delta
  x)^{2}}+\phantom{|}\mu_{i}(\lambda_{i})\phantom{|}\frac{1}{\Delta
  x}\boldsymbol{1}_{(-\infty,0)}(\mu_{i}(\lambda_{i})),
  \nonumber \\
  a_{i,i,i}(\boldsymbol{\lambda})
  & =+\frac{1}{2}\sigma_{i}(\lambda_{i})^{2}\frac{2}{(\Delta
  x)^{2}}+|\mu_{i}(\lambda_{i})|\frac{1}{\Delta x}+\eta_{i},
  & \text{if }0<i<M \tag{\theequation a}
  \label{eq:tridiagonal_tensor} \\
  2a_{i,i,i+1}(\boldsymbol{\lambda})
  & =-\frac{1}{2}\sigma_{i}(\lambda_{i})^{2}\frac{1}{(\Delta
  x)^{2}}-\phantom{|}\mu_{i}(\lambda_{i})\phantom{|}\frac{1}{\Delta
  x}\boldsymbol{1}_{(0,+\infty)}(\mu_{i}(\lambda_{i})),
  \nonumber \\
  a_{i,i+1,i}(\boldsymbol{\lambda})
  & =a_{i,i,i+1}(\boldsymbol{\lambda}),
  \nonumber
\end{align}
and
\begin{equation}
  a_{i,i,i}(\boldsymbol{\lambda}) = 1, \qquad \text{if }i=0,M.
  \tag{\theequation b}
  \label{eq:tridiagonal_tensor2}
\end{equation}
Lastly, define the vector $b=(b_{0},\ldots,b_{M})$ by
\begin{equation}
  b_{i}=
  \begin{cases}
    {\displaystyle \frac{1}{2}\frac{\beta_{i}^{2}}{\alpha_{i}}}, &
    \text{if }0<i<M\\
    g_{i}^{2}, & \text{if }i=0,M.
  \end{cases}\label{eq:od_vector}
\end{equation}
Then, \eqref{eq:discrete_equations} is equivalent to the order $m=3$
Bellman equation
\begin{equation}
  \min_{\boldsymbol{\lambda}\in\boldsymbol{\Lambda}_{\Delta
  x}}A(\boldsymbol{\lambda})u^{2}=b.\label{eq:bellman_3}
\end{equation}

We would like to apply policy iteration to compute a positive solution
of \eqref{eq:bellman_3}. According to \cref{thm:bellman}, this
requires $A(\boldsymbol{\lambda})$ to be an s.d.d. strong M-tensor.
We establish this by showing that $A(\boldsymbol{\lambda})$ is an
s.d.d. Z-tensor with positive diagonals and applying \cref{prop:zhang_0}.
Indeed, that $A(\boldsymbol{\lambda})$ is a Z-tensor with positive
diagonals is clear from its definition. Next, note that by
\eqref{eq:tridiagonal_tensor} and \eqref{eq:tridiagonal_tensor2},
\begin{equation}
  a_{i\cdots
  i}(\boldsymbol{\lambda})+\sum_{(i_{2},\ldots,i_{m})\neq(i,\ldots,i)}a_{ii_{2}\cdots
  i_{m}}(\boldsymbol{\lambda})=
  \begin{cases}
    \eta_{i}, & \text{if }0<i<M\\
    1, & \text{if }i=0,M.
  \end{cases}\label{eq:sum}
\end{equation}
Since $\eta_{i}>0$ by \ref{enu:functions}, the above implies that
$A(\boldsymbol{\lambda})$ is s.d.d.

While the above establishes that the numerical solution is well-defined
for each $\Delta x>0$, we have yet to relate its limit as $\Delta x\downarrow0$
back to the differential equation \eqref{eq:pde}. In order to do
so, we use the framework of viscosity solutions \cite{MR1118699}:
\begin{theorem}
  Suppose \ref{enu:sets}, \ref{enu:functions},
  and that the differential equation \eqref{eq:pde} satisfies a strong
  comparison principle in the sense of Barles and Souganidis \cite{MR1115933}.
  For each $\Delta x>0$, define the function $u_{\Delta
  x}:\overline{\Omega}\rightarrow\mathbb{R}$
  by
  \[
    u_{\Delta x}(x)=\sum_{i=0}^{M}u_{i}\boldsymbol{1}_{[x_{i}-\Delta
    x/2,x_{i}+\Delta x/2)}(x)
  \]
  where $u=(u_{0},\ldots,u_{M})$ is the unique positive solution of
  \eqref{eq:bellman_3}. Then,
  \begin{equation}
    u_{\Delta
    x}\leq\sup_{x}\max\left(\sqrt{\frac{1}{2}\frac{\beta^{2}(x)}{\alpha(x)\eta(x)}},g(x)\right)<\infty\label{eq:stable}
  \end{equation}
  and, as $\Delta x\downarrow0$, $u_{\Delta x}$ converges uniformly
  to the viscosity solution $U$ of \eqref{eq:pde}.
\end{theorem}
We use ``strong'' to mean that comparison holds for the Dirichlet
problem on $\overline{\Omega}$, with the boundary condition included.
This allows us to conclude local uniform convergence relative to
all of $\overline{\Omega}$, rather than only in the interior.
Since $\overline{\Omega}$ is compact, this convergence is uniform.
\begin{proof}
  We first prove the bound \eqref{eq:stable}. Choose $j$ such that
  $u_{j}=\max_{i}u_{i}$. If $0<j<M$, it is straightforward to show
  that, $(L_{\Delta x}^{\lambda}u)_{j}\leq0$. In this case,
  \[
    0=-\max_{\lambda\in\Lambda_{\Delta x}}\left\{ (L_{\Delta
      x}^{\lambda}u)_{j}-\eta_{j}u_{j}+{\displaystyle
    \frac{1}{2}\frac{\beta_{j}^{2}}{\alpha_{j}}\frac{1}{u_{j}}}\right\}
    \geq\eta_{j}u_{j}-{\displaystyle
    \frac{1}{2}\frac{\beta_{j}^{2}}{\alpha_{j}}\frac{1}{u_{j}}}
  \]
  and hence by \ref{enu:functions},
  \[
    u_{j}\leq\sqrt{\frac{1}{2}\frac{\beta_{j}^{2}}{\alpha_{j}\eta_{j}}}\leq\sup_{x}\sqrt{\frac{1}{2}\frac{\beta^{2}(x)}{\alpha(x)\eta(x)}}.
  \]
  If instead $j=0$ or $j=M$, then $u_{j}=g_{j}\leq\max\{g(0),g(1)\}$
  by \ref{enu:functions}. Therefore,
  \[
    u_{j}\leq\sup_{x}\max\left(\sqrt{\frac{1}{2}\frac{\beta^{2}(x)}{\alpha(x)\eta(x)}},g(x)\right),
  \]
  which is finite due to \ref{enu:functions} and the compactness of
  $\overline{\Omega}$.

  Similarly, applying the above argument to an index $j$ such that
  $u_j=\min_i u_i$ gives
  \[
    u_j\geq\inf_x\min\left(\sqrt{\frac{1}{2}\frac{\beta^2(x)}{\alpha(x)\eta(x)}},g(x)\right)>0.
  \]
  Thus $u_{\Delta x}$ is uniformly bounded away from zero, so the
  singular term $1/u_i$ is evaluated only away from zero.

  The remainder of the proof relies on the standard machinery of Barles
  and Souganidis \cite{MR1115933}, so we simply sketch the ideas. The
  discrete equations \eqref{eq:od} define a scheme that is monotone
  and consistent in the sense of \cite{MR1115933}. Moreover, $u_{\Delta x}$
  is bounded independently of $\Delta x$ by \eqref{eq:stable}. Therefore,
  by \cite[Theorem 2.1]{MR1115933}, $u_{\Delta x}$ converges locally
  uniformly to $U$.
\end{proof}

\subsection{\label{subsec:results}Numerical results}

\begin{table}[t]
  \begin{centering}
    \subfloat[Optimize then discretize]{
      \begin{centering}
        {\footnotesize%
          \begin{tabular}{cccccccc}
            \toprule
            \multirow{1}{*}{$M$} & $\phantom{K}$ & Value & Rel. err.
            & Ratio & Its. & Inner its. & Time\tabularnewline
            \midrule
            32 &  & 2.8093 & 1.3e$-$2 &  & 5 & 7 & 4.6e$-$2\tabularnewline
            64 &  & 2.8278 & 6.0e$-$3 &  & 5 & 7 & 4.8e$-$2\tabularnewline
            128 &  & 2.8367 & 2.9e$-$3 & 2.07 & 5 & 7 & 6.5e$-$2\tabularnewline
            256 &  & 2.8411 & 1.3e$-$3 & 2.04 & 5 & 7 & 1.2e$-$1\tabularnewline
            512 &  & 2.8433 & 5.7e$-$4 & 2.02 & 5 & 7 & 1.8e$-$1\tabularnewline
            1024 &  & 2.8444 & 1.9e$-$4 & 2.00 & 5 & 7 & 3.5e$-$1\tabularnewline
            \bottomrule
        \end{tabular}}
        \par
      \end{centering}
    }
    \par
  \end{centering}
  \begin{centering}
    \subfloat[Discretize then optimize]{
      \begin{centering}
        {\footnotesize%
          \begin{tabular}{cccccccc}
            \toprule
            \multirow{1}{*}{$M$} & $K$ & Value & Rel. err. & Ratio &
            Its. & $\phantom{\text{Inner its.}}$ & Time\tabularnewline
            \midrule
            32 & 1 & 1.1783 & 5.9e$-$1 &  & 7 &  & 3.3e$-$2\tabularnewline
            64 & 2 & 1.9179 & 3.3e$-$1 &  & 7 &  & 5.4e$-$2\tabularnewline
            128 & 4 & 2.7161 & 4.5e$-$2 & 0.93 & 7 &  & 1.4e$-$1\tabularnewline
            256 & 8 & 2.7825 & 2.2e$-$2 & 12.0 & 7 &  & 4.5e$-$1\tabularnewline
            512 & 16 & 2.8306 & 5.0e$-$3 & 1.38 & 7 &  & 1.4e$+$0\tabularnewline
            1024 & 32 & 2.8421 & 9.8e$-$4 & 4.19 & 7 &  &
            5.3e$+$0\tabularnewline
            \bottomrule
        \end{tabular}}
        \par
      \end{centering}
    }
    \par
  \end{centering}
  \caption{Convergence results (parameters used:
  \eqref{eq:parameterization})\label{tab:results}}
\end{table}

\begin{figure}[t]
  \begin{centering}
    \begin{center}
      \subfloat[$M=128$]{
        \begin{centering}
          \includegraphics[width=\figwidth]{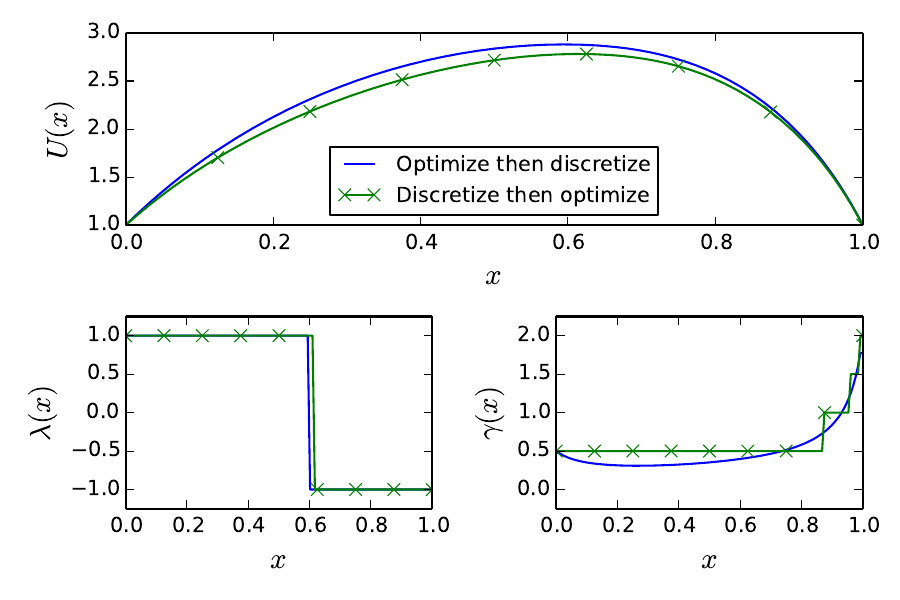}
          \par
        \end{centering}
      }
      \par
    \end{center}

    \begin{center}
      \subfloat[$M=512$]{
        \begin{centering}
          \includegraphics[width=\figwidth]{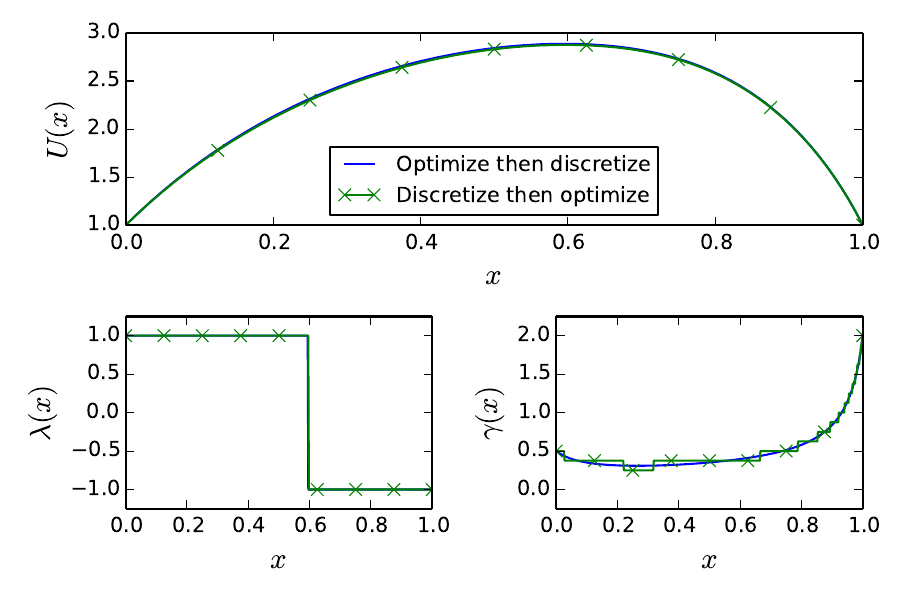}
          \par
        \end{centering}
      }
      \par
    \end{center}
    \par
  \end{centering}
  \caption{Solution and controls computed by both schemes (parameters
  used: \eqref{eq:parameterization})\label{fig:plots}}
\end{figure}

In this subsection, we apply the OD and DO schemes (described in the
previous subsection and \cref{app:do}, respectively) to compute a
numerical solution of \eqref{eq:pde} under the parameterization
\begin{align}
  \sigma(x,\lambda) & =0.2 & \alpha(x) & =2-x & \eta(x) & =0.04\nonumber \\
  \mu(x,\lambda) & =0.04\lambda & \beta(x) & =1+x & g(x) &
  =1\label{eq:parameterization}
\end{align}
where $\lambda\in\Lambda=\{-1,1\}$. Since $\Lambda$ is finite, we
take $\Lambda_{\Delta x}=\Lambda$.
For the DO scheme, we take $\gamma_{\max}=2$ and discretize
$\Gamma_{0}=[0,\gamma_{\max}]$
by a uniform partition $0=\gamma_{0}<\cdots<\gamma_{K}=\gamma_{\max}$
(see \cref{app:do} for details).

In our implementation of {\sc Policy-Iteration}, instead of using the
terminating condition on line \ref{line:equality} of the algorithm,
we terminate the algorithm when it meets the relative error tolerance
\begin{equation}
  \left\Vert u_{(k)}-u_{(k-1)}\right\Vert
  _{\infty}\leq10^{-12}+10^{-6}\left\Vert u_{(k)}\right\Vert _{\infty}.
  \label{eq:relative_error}
\end{equation}
In the case of the DO scheme, $m=2$ and $A(P)$ is a tridiagonal
matrix (see \eqref{eq:tridiagonal} and \eqref{eq:tridiagonal2} of
\cref{app:do}). Therefore,
we use a tridiagonal solver to solve $A(P)x=b(P)$. As for the OD
scheme, $m=3$ and we use the Newton's method described in
\cite[Section 4]{MR3519197}
to solve $A(P)x^{2}=b(P)$. Denoting by $x_{(k)}$ the iterates produced
by Newton's method, we terminate the algorithm when it meets the error tolerance
\[
  \left\Vert x_{(k)}-x_{(k-1)}\right\Vert
  _{\infty}\leq10^{-24}+10^{-12}\left\Vert x_{(k)}\right\Vert _{\infty}.
\]

Convergence results are given in \cref{tab:results}, in which
we report a representative value of the numerical solution (Value),
the relative error (Rel. err.), ratio of errors (Ratio), number of
policy iterations (Its.), average number of iterations (Inner its.)
to solve the system $A(P)x^{m-1}=b(P)$ if applicable, and total time
elapsed in seconds (Time). The representative value of the numerical
solution is $u_{\Delta x}(x_{0})$ where $x_{0}=\nicefrac{1}{2}$
is the midpoint of $\overline{\Omega}$. The relative error is given
by
\[
  \left|\frac{u_{\Delta x}(x_{0})-U(x_{0})}{U(x_{0})}\right|
\]
where $U$ is the exact solution. Since the exact solution is generally
unavailable, we replace $U$ by the solution computed by the OD scheme
at a level of refinement higher than that which is shown in the table.
The ratio of errors is given by
\[
  \frac{u_{\Delta x/2}(x_{0})-u_{\Delta x}(x_{0})}{u_{\Delta
  x/4}(x_{0})-u_{\Delta x/2}(x_{0})}
\]
so that the base-2 logarithm of this quantity gives us an estimate
on the order of convergence (e.g., $\text{Ratio}\approx2$ suggests
linear convergence, $\text{Ratio}\approx4$ suggests quadratic, etc.).
Plots of the solution and optimal controls are given in \cref{fig:plots}.

The OD scheme is faster and more accurate than the DO scheme. Since
both schemes require roughly the same number of policy iterations, it
follows that the DO scheme loses most of its time on the pivot step
on line \ref{line:policy_improvement} of {\sc Policy-Iteration}.
As $K\rightarrow\infty$, this effect becomes more pronounced. Note
also that the OD scheme exhibits a fairly stable linear order of
convergence while that of the DO scheme is erratic.

\subsection{\label{subsec:irregular}A problem which is neither s.d.d.
nor weakly irreducibly diagonally dominant}

\begin{table}[t]
  \begin{centering}
    \subfloat[Optimize then discretize]{
      \begin{centering}
        {\footnotesize%
          \begin{tabular}{cccccccc}
            \toprule
            \multirow{1}{*}{$M$} & $\phantom{K}$ & Value & Rel. err.
            & Ratio & Its. & Inner. its. & Time\tabularnewline
            \midrule
            32 &  & 3.0703 & 1.6e$-$1 &  & 3 & 6.67 & 4.3e$-$2\tabularnewline
            64 &  & 3.3567 & 8.5e$-$2 &  & 3 & 6.67 & 3.7e$-$2\tabularnewline
            128 &  & 3.5114 & 4.3e$-$2 & 1.85 & 3 & 6.67 &
            5.4e$-$2\tabularnewline
            256 &  & 3.5917 & 2.1e$-$2 & 1.92 & 3 & 6.67 &
            9.3e$-$2\tabularnewline
            512 &  & 3.6327 & 9.9e$-$3 & 1.96 & 3 & 6.67 &
            1.7e$-$1\tabularnewline
            1024 &  & 3.6534 & 4.3e$-$3 & 1.98 & 3 & 6.67 &
            3.3e$-$1\tabularnewline
            2048 &  & 3.6638 & 1.4e$-$3 & 1.99 & 3 & 7.67 &
            6.5e$-$1\tabularnewline
            \bottomrule
        \end{tabular}}
        \par
      \end{centering}
    }
    \par
  \end{centering}
  \begin{centering}
    \subfloat[Discretize then optimize]{
      \begin{centering}
        {\footnotesize%
          \begin{tabular}{cccccccc}
            \toprule
            \multirow{1}{*}{$M$} & $K$ & Value & Rel. err. & Ratio &
            Its. & $\phantom{\text{Inner its.}}$ & Time\tabularnewline
            \midrule
            32 & 1 & 0.9273 & 7.5e$-$1 &  & 3 &  & 1.7e$-$2\tabularnewline
            64 & 2 & 1.8839 & 4.9e$-$1 &  & 5 &  & 3.9e$-$2\tabularnewline
            128 & 4 & 3.2430 & 1.2e$-$1 & 0.70 & 7 &  & 1.4e$-$1\tabularnewline
            256 & 8 & 3.5376 & 3.6e$-$2 & 4.61 & 9 &  & 6.4e$-$1\tabularnewline
            512 & 16 & 3.6163 & 1.4e$-$2 & 3.74 & 14 &  &
            3.5e$+$0\tabularnewline
            1024 & 32 & 3.6490 & 5.5e$-$3 & 2.41 & 15 &  &
            1.4e$+$1\tabularnewline
            2048 & 64 & 3.6629 & 1.7e$-$3 & 2.35 & 7 &  &
            2.3e$+$1\tabularnewline
            \bottomrule
        \end{tabular}}
        \par
      \end{centering}
    }
    \par
  \end{centering}
  \caption{Convergence results (parameters used:
  \eqref{eq:parameterization_2})\label{tab:results_2}}
\end{table}

\begin{figure}[t]
  \begin{centering}
    \begin{center}
      \includegraphics[width=\figwidth]{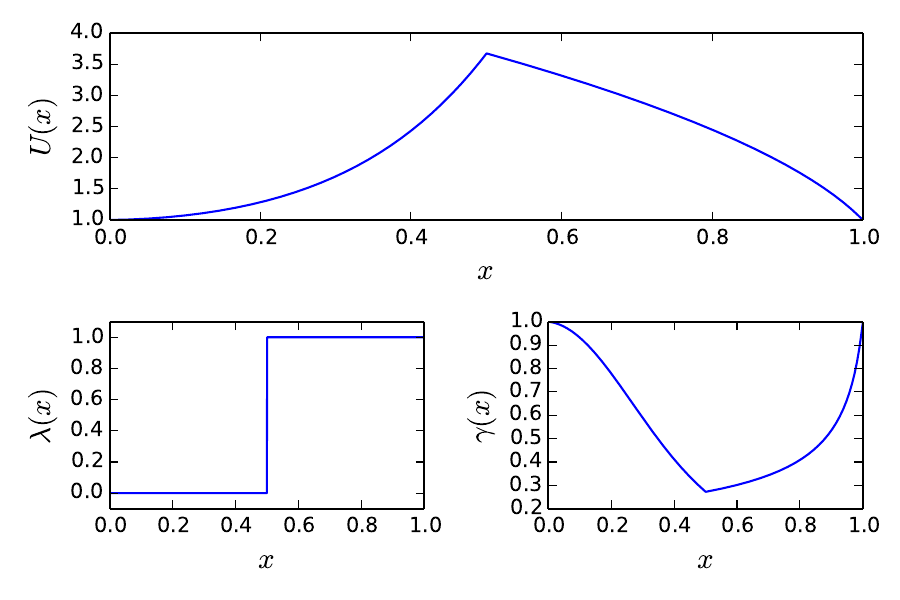}
      \par
    \end{center}
    \par
  \end{centering}
  \caption{Solution and controls computed (parameters used:
      \eqref{eq:parameterization_2};
  OD scheme only)\label{fig:plots_2}}
\end{figure}

\begin{figure}
  \centering
  \scalebox{\figscale}{
    \begin{tikzpicture}[node distance=1.25cm]
      \node [thick, draw=black!35, ellipse, fill=black!15
      ] (a) {$0$};
      \node [thick, draw=black!35, ellipse, fill=black!15, right of=a
      ] (b) {$1$};
      \node [thick,                ellipse,                right of=b
      ] (c) {$\cdots$};
      \node [thick, draw=black!35, ellipse, fill=black!15, right
      of=c, node distance=1.5cm] (d) {$i{-}1$};
      \node [thick, draw=black!35, ellipse, fill=black!15, right of=d
      ] (e) {$i$};
      \node [thick, draw=black!35, ellipse, fill=black!15, right of=e
      ] (f) {$i{+}1$};
      \node [thick,                ellipse,                right of=f
      ] (g) {$\cdots$};
      \node [thick, draw=black!35, ellipse, fill=black!15, right
      of=g, node distance=1.5cm] (h) {$M{-}1$};
      \node [thick, draw=black!35, ellipse, fill=black!15, right
      of=h, node distance=1.5cm] (i) {$M$};
      \path [thick, ->, >=stealth, dashed] (b.south west) edge [bend
      left] (a.south east);
      \path [thick, ->, >=stealth, dashed] (c.south west) edge [bend
      left] (b.south east);
      \path [thick, ->, >=stealth, dashed] (d.south west) edge [bend
      left] (c.south east);
      \path [thick, ->, >=stealth, dashed] (e.south west) edge [bend
      left] (d.south east);
      \path [thick, ->, >=stealth, dashed] (f.south west) edge [bend
      left] (e.south east);
      \path [thick, ->, >=stealth, dashed] (g.south west) edge [bend
      left] (f.south east);
      \path [thick, ->, >=stealth, dashed] (h.south west) edge [bend
      left] (g.south east);
      \path [thick, ->, >=stealth        ] (b.north east) edge [bend
      left] (c.north west);
      \path [thick, ->, >=stealth        ] (c.north east) edge [bend
      left] (d.north west);
      \path [thick, ->, >=stealth        ] (d.north east) edge [bend
      left] (e.north west);
      \path [thick, ->, >=stealth        ] (e.north east) edge [bend
      left] (f.north west);
      \path [thick, ->, >=stealth        ] (f.north east) edge [bend
      left] (g.north west);
      \path [thick, ->, >=stealth        ] (g.north east) edge [bend
      left] (h.north west);
      \path [thick, ->, >=stealth        ] (h.north east) edge [bend
      left] (i.north west);
    \end{tikzpicture}
  }

  \caption{Directed graph of $A$ defined by
    \eqref{eq:tridiagonal_tensor} and \eqref{eq:tridiagonal_tensor2}
    with $\sigma$ and $\mu$ given by
  \eqref{eq:parameterization_2}\label{fig:graph_2}}
\end{figure}
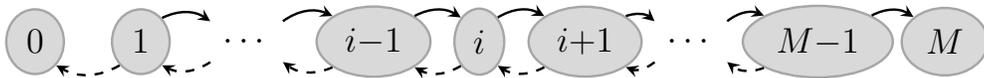

We now turn to a parameterization of \eqref{eq:pde} whose corresponding
``discretization tensor'' $A(\boldsymbol{\lambda})$ defined by
\eqref{eq:tridiagonal_tensor} and \eqref{eq:tridiagonal_tensor2} is
neither s.d.d. nor weakly irreducibly
diagonally dominant. We are motivated by an analogous phenomenon that
occurs for classical discretizations of degenerate elliptic differential
equations first studied in \cite{MR0162367} in which the matrix arising
from the discretization is neither s.d.d. nor irreducibly diagonally
dominant (cf. \cite{MR1384509,MR3493959,chen2018monotone}).

The parameterization we study is
\begin{align}
  \sigma(x,\lambda) & =0.3(1-\lambda) & \alpha(x) & =1 & \eta(x) &
  =\boldsymbol{1}_{\{x\leq0.5\}}(x)\nonumber \\
  \mu(x,\lambda) & =0.04\lambda & \beta(x) & =1 & g(x) &
  =1\label{eq:parameterization_2}
\end{align}
where $\lambda\in\Lambda=\{0,1\}$.
Note that this parameterization does not satisfy \ref{enu:functions}
since $\eta$
is discontinuous and allowed to be zero.
Therefore, the argument used to establish that
$A(\boldsymbol{\lambda})$ is a strong M-tensor in \cref{subsec:od}
fails (see, in particular, \eqref{eq:sum}).

Letting $\omega$ be any function which maps $\mathbb{R}_{++}$ to
itself such that $\lim_{t\rightarrow 0}\omega(t)=0$, one way to get
around the above issue is to replace $A(\boldsymbol{\lambda})$ by
$A(\boldsymbol{\lambda}) + I\omega({\Delta x})$ in
\eqref{eq:bellman_3}, since the latter is trivially s.d.d.
The obvious downside of this approach is that it introduces
additional discretization error.

Fortunately, it turns out that we can directly establish that
$A(\boldsymbol{\lambda})$ is a strong M-tensor by relying on the
theory of w.c.d.d. tensors.
In particular, by \eqref{eq:sum},
\begin{enumerate}[label=(\emph{\roman*})]
  \item $A(\boldsymbol{\lambda})$ is a w.d.d. (not s.d.d.) Z-tensor
    since $\eta_{i}\geq0$ and
  \item $0,M\in J(A(\boldsymbol{\lambda}))$ are s.d.d. rows.
\end{enumerate}
The directed graph of $A(\boldsymbol{\lambda})$ is shown in
\cref{fig:graph_2} where we have
\begin{enumerate}[label=(\emph{\roman*})]
  \item ignored self-loops of the form $i\rightarrow i$ and
  \item used a dashed line to indicate an edge that is present only
    when $\lambda_{i}=0$.
\end{enumerate}
Note that for any $i\notin J(A(\boldsymbol{\lambda}))$, we can form
the walk $i\rightarrow i+1\rightarrow\cdots\rightarrow M$ ending at
the s.d.d. vertex $M\in J(A(\boldsymbol{\lambda}))$.
Therefore, $A(\boldsymbol{\lambda})$ is w.c.d.d. and hence a strong
M-tensor by \cref{thm:wcdd}.
Now, by \cref{thm:bellman}, we can compute a solution of the OD scheme
as applied to the parameterization \eqref{eq:parameterization_2}
by policy iteration. Convergence results and plots are shown in
\cref{tab:results_2}
and \cref{fig:plots_2}, respectively. The $C^{1}$ discontinuity
in the solution is due to the discontinuity in $\eta$.

\section{Summary}

In this work, we introduced the high order Bellman equation
\cref{eq:bellman}, extending classical Bellman equations to the tensor setting.
We also introduced w.c.d.d. tensors (\cref{def:wcdd}), also extending
the notion of w.c.d.d. matrices to the tensor setting.
We established a relationship between w.c.d.d. tensors and M-tensors
(\cref{thm:wcdd}), analogous to the relationship between w.c.d.d.
matrices and M-matrices \cite{azimzadeh2017fast}.
We proved that a sufficient condition to ensure the existence and
uniqueness of a positive solution to a high order Bellman equation is
that the tensors appearing in the equation are s.d.d. M-tensors
(\cref{lem:uniqueness} and \cref{lem:existence}).
We also showed that the s.d.d. requirement can be relaxed to the
weaker requirement of w.c.d.d. so long as we restrict ourselves to a
finite set of policies (\cref{lem:existence_wcdd}).
In this case, the solution of \cref{eq:bellman} can be computed by a
policy iteration algorithm (\cref{thm:bellman}).
The question of whether or not the assumption of finitude can be
removed remains open (\cref{rem:finite_disclaimer}).
We applied our findings to create a so-called ``optimize then
discretize'' scheme for an optimal stochastic control problem which
outperforms (in both computation time and accuracy) a classical
``discretize then optimize'' approach (\cref{sec:application}).

\bibliography{main}
\appendix \section{\label{app:do}Discretize then optimize scheme}

In this appendix, we derive the DO scheme for the differential
equation \eqref{eq:pde}.
Since the set $\Gamma=[0,\infty)$ is unbounded, we restrict our
attention to controls $\gamma$ in some bounded interval
$\Gamma_{0}=[0,\gamma_{\max}]$.
To ensure the consistency of the resulting scheme, $\gamma_{\max}$
should be chosen sufficiently large.
For each $\Delta x>0$, we let $\Gamma_{\Delta x}$ denote a finite
subset of $\Gamma_{0}$ such that $d_{H}(\Gamma_{\Delta
x},\Gamma_{0})\rightarrow0$ as $\Delta x\downarrow0$.
The DO discretization is given by the $M+1$ equations
\begin{equation}
  \begin{cases}
    -{\displaystyle \max_{(\gamma,\lambda)\in\Gamma_{\Delta
    x}\times\Lambda_{\Delta x}}}\left\{ (L_{\Delta
    x}^{\lambda}u)_{i}-\eta_{i}u_{i}-\frac{1}{2}\alpha_{i}\gamma^{2}u_{i}+\beta_{i}\gamma\right\}
    =0, & \text{for }0<i<M\\
    u_{i}=g_{i} & \text{for }i=0,M
  \end{cases}\label{eq:do}
\end{equation}
where the various quantities $L_{\Delta x}^{\lambda}$, $\eta_{i}$,
etc. are defined in \cref{subsec:od} (compare \eqref{eq:do} with the
OD discretization \eqref{eq:od}).

We can transform \eqref{eq:do} into a classical $m=2$ Bellman
equation as follows.
Define the Cartesian product $\boldsymbol{\Gamma}_{\Delta x} =
(\Gamma_{\Delta x})^{M+1}$ and denote by
$\boldsymbol{\gamma}=(\gamma_{0},\ldots,\gamma_{M})$ an element of
$\boldsymbol{\Gamma}_{\Delta x}$ with $\gamma_{i}\in\Gamma_{\Delta x}$.
Define $\boldsymbol{\Lambda}_{\Delta x}$ and
$\boldsymbol{\lambda}=(\lambda_{0},\ldots,\lambda_{M})$ similarly.
Let
$A(\boldsymbol{\gamma},\boldsymbol{\lambda})=(a_{ij}(\boldsymbol{\gamma},\boldsymbol{\lambda}))$
be the tridiagonal matrix whose only nonzero entries are
\stepcounter{equation}
\begin{align}
  a_{i,i-1}(\boldsymbol{\gamma},\boldsymbol{\lambda}) &
  =-\frac{1}{2}\sigma_{i}(\lambda_{i})^{2}\frac{1}{(\Delta
  x)^{2}}+\phantom{|}\mu_{i}(\lambda_{i})\phantom{|}\frac{1}{\Delta
  x}\boldsymbol{1}_{(-\infty,0)}(\mu_{i}(\lambda_{i})), \nonumber \\
  a_{i,i}(\boldsymbol{\gamma},\boldsymbol{\lambda}) &
  =+\frac{1}{2}\sigma_{i}(\lambda_{i})^{2}\frac{2}{(\Delta
  x)^{2}}+|\mu_{i}(\lambda_{i})|\frac{1}{\Delta
  x}+\eta_{i}+\frac{1}{2}\alpha_{i}\gamma_{i}^{2},
  & \text{if }0<i<M \tag{\theequation a} \label{eq:tridiagonal} \\
  a_{i,i+1}(\boldsymbol{\gamma},\boldsymbol{\lambda}) &
  =-\frac{1}{2}\sigma_{i}(\lambda_{i})^{2}\frac{1}{(\Delta
  x)^{2}}-\phantom{|}\mu_{i}(\lambda_{i})\phantom{|}\frac{1}{\Delta
  x}\boldsymbol{1}_{(0,+\infty)}(\mu_{i}(\lambda_{i})), \nonumber
\end{align}
and
\begin{equation}
  a_{i,i}(\boldsymbol{\gamma},\boldsymbol{\lambda}) =1, \qquad
  \text{if }i=0,M. \tag{\theequation b} \label{eq:tridiagonal2}
\end{equation}
Lastly, define the vector
$b(\boldsymbol{\gamma})=(b_{0}(\boldsymbol{\gamma}),\ldots,b_{M}(\boldsymbol{\gamma}))$
by
\[
  b_{i}(\boldsymbol{\gamma})=
  \begin{cases}
    {\displaystyle \beta}_{i}\gamma_{i}, & \text{if }0<i<M\\
    g_{i}, & \text{if }i=0,M.
  \end{cases}
\]
Then, the discrete equations \eqref{eq:do} are equivalent to the
classical Bellman equation
\[
  \min_{(\boldsymbol{\gamma},\boldsymbol{\lambda})\in\boldsymbol{\Gamma}_{\Delta
  x}\times\boldsymbol{\Lambda}_{\Delta x}}\left\{
  A(\boldsymbol{\gamma},\boldsymbol{\lambda})u-b(\boldsymbol{\gamma})\right\}
  =0.
\]
The downside of the above approach is twofold:
\begin{enumerate}[label=(\emph{\roman*})]
  \item The size of the policy set is
    $|\mathcal{P}|=|\boldsymbol{\Gamma}_{\Delta
    x}||\boldsymbol{\Lambda}_{\Delta x}|$.
    In the OD approach, the size of the policy set is
    $|\mathcal{P}|=|\boldsymbol{\Lambda}_{\Delta x}|$ (recall that
      the policy iteration algorithm takes, in the worst case,
    $|\mathcal{P}|$ iterations).
  \item Assuming $\gamma_{\max}$ is chosen large enough, the approximation
    of $\Gamma_{0}$ by $\Gamma_{\Delta x}$ introduces
    $O(d_{H}(\Gamma_{\Delta x},\Gamma_{0}))$
    local truncation error.
\end{enumerate}

\end{document}